\documentclass[smallextended,leqno,final]{article}
\usepackage[english]{babel}
\usepackage[latin1]{inputenc}
\usepackage{amsmath, amsfonts, amssymb,amsthm}
\usepackage{mathptmx}
\usepackage{url}
\usepackage[backref=page]{hyperref}
\hypersetup{pdfpagemode=UseNone, pdfstartview={FitH}}
\usepackage{cleveref}
\newtheorem{theorem}{Theorem}
\newtheorem{remark}{Remark}
\newtheorem{problem}{Problem}
\newtheorem{definition}{Definition}
\newtheorem{proposition}{Proposition}
\newtheorem{lemma}{Lemma} 
\newtheorem{corollary}{Corollary}

\crefname{problem}{Problem}{Problems}
\Crefname{problem}{Problem}{Problems}
\newcommand{\lag}{\left \langle}
\newcommand{\rog}{\right \rangle}
\newcommand{\df}{\mathrel{\mathop:}=}
\crefname{problem}{Problem}{Problems}
\Crefname{problem}{Problem}{Problems}
\begin{document}
\date{}
\title{First-order Nilpotent Minimum Logics: first steps}
\author{Matteo Bianchi\\{\small Department of Mathematics ``Federigo Enriques''}\\{\small Università degli Studi di Milano}\\{\footnotesize\texttt{\href{mailto:matteo.bianchi@unimi.it}{matteo.bianchi@unimi.it}}}}
\maketitle
\begin{abstract}
Following the lines of the analysis done in \cite{fgod,mgod} for first-order Gödel logics, we present an analogous investigation for Nilpotent Minimum logic NM.  
We study decidability and reciprocal inclusion of various sets of first-order tautologies of some subalgebras of the standard Nilpotent Minimum algebra. We establish a connection between the validity in an NM-chain of certain first-order formulas and its order type. Furthermore, we analyze axiomatizability, undecidability and the monadic fragments. 
\end{abstract}
\section{Introduction}
Nilpotent Minimum logic (NM) is a many-valued logic firstly introduced in \cite{mtl}: its name is due to the fact that NM is complete w.r.t. the algebraic structure $[0,1]_\text{NM}=\lag [0,1], *,\Rightarrow, \min, \max, 0, 1\rog$, with $*,\Rightarrow$ being Nilpotent Minimum t-norm and its residuum. Triangular norms (t-norms) are particular types of functions, originally introduced in the context of probabilistic metric spaces (see \cite{prob}), that can be used to give the semantics for the conjunction connective of a many-valued logic (see \cite{kmp} for details about t-norms and residua): for every continuous t-norm the associated residuum is an operation (that can be argued to be) suitable to give the semantics for an implication connective.

In 1998, P. Hájek wrote the monograph \cite{haj}, mainly devoted to a family of many-valued logics that is strictly connected to continuous t-norms and their residua: in this framework are also included the famous \L ukasiewicz and Gödel logics (in the following we will discuss more in detail about this last one). 

Nilpotent Minimum t-norm, introduced in \cite{fod}, was an example (probably the first one) of non-continuous but left-continuous t-norm. Left-continuous t-norms are particularly important, since a t-norm admits a residuum if and only if it is left-continuous (\cite{beg}): this fact stimulated analysis like \cite{mtl}, where the Monoidal t-norm based logic (MTL) was introduced, by showing that it is the logic of all left-continuous t-norms and their residua. The logic MTL is at the basis of an ample hierarchy of logics, that includes the ones introduced in \cite{haj}, as well as Nilpotent Minimum logic, that was presented in \cite{mtl}. So the birth of NM was essentially due to technical reasons: in particular, it was presented as an exemplification of a logic (of this framework) associated to a left-continuous t-norm. 
However, it is worth to point out that this logic and the corresponding algebraic semantics offer many interesting features, and during the years, NM and its corresponding variety have been studied under numerous aspects. For example:
\begin{itemize}
\item Combinatorial aspects and description of the free algebras \cite{agm,abm,bus}.
\item States and connection with probability theory (\cite{ag}).
\item Computational complexity for satisfiability and tautologicity problems (\cite{nmcomp}).
\item Connections with others non-classical logics: for example, Nelson's constructive logic (\cite{bc}). 
\item Extensions with truth constants, in the propositional and in the first-order case (\cite{egnfirst,egn,egn1,egn2}).
\item Alternative semantics. This is a (joint) work in progress, but it is possible to give a temporal like semantics (on the line of \cite{tmpg,tbl}) to NM logic. This could be useful to show some other aspects and characteristics of this logic.
\end{itemize}
As previously argued, NM has some relation with Nelson's logic: however, it is also connected with a famous many-valued (and superintuitionistic) logic, namely Gödel logic.

Gödel logic (G) was introduced in \cite{dum} by taking inspiration from a paper by Kurt Gödel (\cite{godint}). This logic was firstly defined as a superintuitionistic logic, but it can also be axiomatized as an axiomatic extension of MTL (see \cite{pmtl,haj}). The algebraic semantics of G is given by a particular class of MTL-algebras (that also forms a subvariety of Heyting algebras), called Gödel-algebras. As pointed out in \cite{pre}, at the propositional level there is only one infinite valued Gödel logic, in the sense that G is complete w.r.t. every infinite totally ordered Gödel-algebra.
In the first-order case G$\forall$, however, the situation is different and there are many infinitely-valued first-order logics: this means that there are many totally ordered Gödel-algebras whose set of first-order tautologies is different w.r.t. the one of G$\forall$.
A deep analysis about first-order Gödel logics, and a general classification has been done in \cite{pre,fgod,mgod} (see also the general survey \cite{presurv}): in particular, in \cite{fgod,mgod} the sets of first-order tautologies associated to the subalgebras of $[0,1]_\text{G}$ (the standard Gödel algebra) have been studied and a general classification about decidability has been provided, also for the monadic fragments.

Which is the previously cited connection, about NM and G?
As shown in \cite{bus}, every NM-chain is isomorphic to the connected or the disconnected rotation (see \cite{bus,jen}) of a Gödel-chain: this shows how strongly related are these two varieties of algebras.
 
We now move back to NM, and to the aim and content of this paper. Whereas the propositional level of NM has been extensively investigated, in the first-order level the situation is more delicate. For example, a systematic analysis like the one done in \cite{fgod,mgod} for Gödel logics is missing, for NM. The aim of this paper is to lay the foundations of a study of this type.

If we consider the logic associated to a totally ordered algebra, then over NM there are only two different infinite-valued logics, at the propositional level: NM and NM$^-$ (see \Cref{rem:nmlog}). At the first-order level, instead, the situation is different: there are infinite totally ordered NM-algebras with negation fixpoint whose set of first-order tautologies is different w.r.t. the one of $[0,1]_\text{NM}$. In this paper we will study the sets of first-order tautologies of some subalgebras of $[0,1]_\text{NM}$: in particular finite NM-chains and other four infinite NM-chains (with and without negation fixpoint). Moreover we will find a connection between the validity, in an NM-chain, of certain first-order formulas and its order type. Finally, we will analyze axiomatization, undecidability and the monadic fragments.

Our investigation has been inspired by the work done for first-order Gödel logic in \cite{fgod} and \cite{mgod}, where a complete classification of the sets of first-order tautologies associated to Gödel-chains (subalgebras of $[0,1]_\text{G}$) has been given. 
Unfortunately, we do not provide here a complete classification for the case of Nilpotent Minimum chains (for this reason the title indicates ``first steps''). The infinite NM-chains discussed in this paper have been chosen essentially for their relations with some particularly important Gödel chains ($G_\uparrow$ and $G_\downarrow$, see the following sections for the definitions): as we will see, this connection will help in the study of the undecidability results. 
\section{Preliminaries}\label{sec:prelim}
\subsection{Syntax}
Nilpotent Minimum logic was introduced in \cite{mtl}, as an extension of the Monoidal t-norm based logic (MTL): this last one is the logic at the base (in the sense that the other logics of this family are obtained by adding axiom to it) of a framework of many-valued logics initially introduced by Hájek in \cite{haj}. MTL was introduced in \cite{mtl}: as shown in \cite{nog} this logic is algebraizable in the sense of \cite{bp} and its equivalent algebraic semantics forms a variety (the variety of MTL-algebras). From the results of \cite{bp,nog} it follows that also every extension of MTL (a logic obtained from it by adding other axioms) is algebraizable in this way.

The language of MTL is based over connectives $\{\&, \land, \to, \bot\}$ (the first three are binary, whilst the last one is $0$-ary). The notion of formula is defined in the usual way. 

Useful derived connectives are the following
\begin{align}
\tag{negation}\neg\varphi\df& \varphi\to\bot\\
\tag{disjunction}\varphi\vee\psi\df& ((\varphi\to\psi)\to\psi)\land((\psi\to\varphi)\to\varphi)\\
\tag{biconditional}\varphi\leftrightarrow\psi\df&(\varphi\to\psi)\land(\psi\to\varphi)\\
\tag{top}\top\df& \neg\bot
\end{align}
For reader's convenience we list the axioms of MTL
\begin{align}
\tag{A1}&(\varphi \rightarrow \psi)\rightarrow ((\psi\rightarrow \chi)\rightarrow(\varphi\rightarrow \chi))\\
\tag{A2}&(\varphi\&\psi)\rightarrow \varphi\\
\tag{A3}&(\varphi\&\psi)\rightarrow(\psi\&\varphi)\\
\tag{A4}&(\varphi\land\psi)\rightarrow \varphi\\
\tag{A5}&(\varphi\land\psi)\rightarrow(\psi\land\varphi)\\
\tag{A6}&(\varphi\&(\varphi\rightarrow \psi))\rightarrow (\psi\land\varphi)\\
\tag{A7a}&(\varphi\rightarrow(\psi\rightarrow\chi))\rightarrow((\varphi\&\psi)\rightarrow \chi)\\
\tag{A7b}&((\varphi\&\psi)\rightarrow \chi)\rightarrow(\varphi\rightarrow(\psi\rightarrow\chi))\\
\tag{A8}&((\varphi\rightarrow\psi)\rightarrow\chi)\rightarrow(((\psi\rightarrow\varphi)\rightarrow\chi)\rightarrow\chi)\\
\tag{A9}&\bot\rightarrow\varphi
\end{align}
As inference rule we have modus ponens:
\begin{equation}
\tag{MP}\frac{\varphi\quad \varphi\rightarrow\psi}{\psi}
\end{equation}
Gödel logic (G) is obtained from MTL by adding
\begin{equation}
\tag{id}\varphi\to\varphi\&\varphi.
\end{equation}
Nilpotent Minimum Logic (NM), introduced in \cite{mtl}, is obtained from MTL by adding the following axioms:
\begin{align}
\tag{involution}&\neg\neg\varphi\to\varphi\\
\tag{WNM}&\neg(\varphi\&\psi)\vee((\varphi\land\psi)\to(\varphi\&\psi)).
\end{align}
The notions of theory, syntactic consequence, proof are defined as usual.
\subsection{Semantics}\label{subsec:sem}
An MTL algebra is an algebra $\lag A,*,\Rightarrow,\sqcap,\sqcup,0,1\rog$ such that
\begin{enumerate}
\item $\lag A,\sqcap,\sqcup, 0,1\rog$ is a bounded lattice with minimum $0$ and maximum $1$.
\item $\lag A,*,1 \rog$ is a commutative monoid.
\item $\lag *,\Rightarrow \rog$ forms a \emph{residuated pair}: $z*x\leq y$ iff $z\leq x\Rightarrow y$ for all $x,y,z\in A$.
\item The following axiom hold, for all $x,y\in A$:
\begin{equation}
\tag{Prelinearity}(x\Rightarrow y)\sqcup(y\Rightarrow x)=1
\end{equation}
A totally ordered MTL-algebra is called MTL-chain.
\end{enumerate}
A G-algebra is an MTL-algebra satisfying
\begin{equation*}
x=x*x.
\end{equation*}
It is well known (see for example \cite{dm,pre}) that in every G-chain the following hold:
\begin{align*}
x*y=&\min(x,y)\\
x\Rightarrow y=&
\begin{cases}
1&\text{if }x\leq y\\
y&\text{Otherwise.}
\end{cases}
\end{align*}
Some examples of G-chains are the following:
\begin{itemize}
\item $G_\uparrow=\lag \{1-\frac{1}{n}:\ n\in\mathbb{N}^+\}\cup\{1\},*,\Rightarrow,\min,\max,0,1\rog$
\item $G_\downarrow=\lag \{\frac{1}{n}:\ n\in\mathbb{N}^+\}\cup \{0\},*,\Rightarrow,\min,\max,0,1\rog$
\item $G_n=\lag \{0,\frac{1}{n-1},\dots,\frac{n-2}{n-1},1\},*,\Rightarrow,\min,\max,0,1\rog$
\item $[0,1]_\text{G}=\lag [0,1],*,\Rightarrow,\min,\max,0,1\rog$
\end{itemize}
In particular it is easy to check (see \cite{pre}) that every finite G-chain of cardinality $n$ is isomorphic to $G_n$.
\paragraph*{}
An NM-algebra is an MTL-algebra that satisfies the following equations:
\begin{align*}
&\sim\sim x=x\\
&\sim(x*y)\sqcup((x\sqcap y)\Rightarrow(x*y))=1
\end{align*}
Where $\sim x$ indicates $x\Rightarrow 0$.

\smallskip
Moreover, as noted in \cite{gis}, in \emph{each} NM-chain the following hold:
\begin{align*}
x*y=&
\begin{cases}
0&\text{if }x\leq n(y)\\
\min(x,y)&\text{Otherwise.}
\end{cases}\\
x\Rightarrow y=&
\begin{cases}
1&\text{if }x\leq y\\
\max(n(x),y)&\text{Otherwise.}
\end{cases}
\end{align*}
Where $n$ is a strong negation function, i.e. $n:A\to A$ is an order-reversing mapping ($x<y$ implies $n(x)>n(y)$) such that $n(0)=1$ and $n(n(x))=x$, for each $x\in A$. Observe that $n(x)=x\Rightarrow 0=\sim x$, for each $x\in A$.

A negation fixpoint is an element $x\in A$ such that $n(x)=x$: note that if this element exists then it must be unique (otherwise $n$ fails to be order-reversing).
A positive element is an $x\in A$ such that $x>n(x)$; the definition of negative element is the dual (substitute $>$ with $<$).

Concerning the finite chains, in \cite{gis} it is shown that two finite NM-chains with the same cardinality are isomorphic (see the remarks after \cite[Proposition 2]{gis}): for this reason we will denote them with $NM_n$, $n$ being an integer greater that $1$.

We now give some examples of infinite NM-chains that will be useful in the following: for all of them the order is given by $\leq_\mathbb{R}$ and $n(x)=1-x$.

\smallskip
\begin{itemize}
\item $NM_\infty=\lag \{\frac{1}{n}:\ n\in\mathbb{N}^+\}\cup \{1-\frac{1}{n}:\ n\in\mathbb{N}^+\},*,\Rightarrow,\min,\max,0,1\rog$
\item $NM_\infty^-=\lag \{\{\frac{1}{n}:\ n\in\mathbb{N}^+\}\cup \{1-\frac{1}{n}:\ n\in\mathbb{N}^+\}\}\setminus\{\frac{1}{2}\},*,\Rightarrow,\min,\max,0,1\rog$
\item $NM'_\infty=\lag \{\frac{1}{2}-\frac{1}{2n}:\ n\in\mathbb{N}^+\}\cup \{\frac{1}{2}+\frac{1}{2n}:\ n\in\mathbb{N}^+\}\cup\{\frac{1}{2}\},*,\Rightarrow,\min,\max,0,1\rog$
\item ${NM'}^-_\infty=\lag \{\frac{1}{2}-\frac{1}{2n}:\ n\in\mathbb{N}^+\}\cup \{\frac{1}{2}+\frac{1}{2n}\}:\ n\in\mathbb{N}^+\},*,\Rightarrow,\min,\max,0,1\rog$
\item $[0,1]_\text{NM}=\lag [0,1],*,\Rightarrow,\min,\max,0,1\rog$
\end{itemize}
In this last case, $*$ is called Nilpotent Minimum t-norm \cite{fod}. Note that the first four chains of the list and every finite NM-chain\footnote{Since two finite NM-chains with the same cardinality are isomorphic, then we can consider $NM_n$ as defined over the set $\{0,\frac{1}{n-1},\dots,\frac{n-1}{n-1}\}$.} are all subalgebras of $[0,1]_\text{NM}$.

The notion of assignment, model, satisfiability and tautology are defined as usual: we refer to \cite{mtl} for details.
\begin{theorem}[\cite{pre,gis}]\label{teo:log}
\begin{itemize}
\item Every infinite Gödel-chain is complete w.r.t. Gödel logic.
\item Every infinite NM-chain with negation fixpoint is complete w.r.t. Nilpotent Minimum logic.
\end{itemize}
\end{theorem}
Concerning the variety of NM-algebras, we have the following result:
\begin{theorem}[{\cite[Theorem 2]{gis}}]\label{teo:fixpointeq}
\begin{enumerate}
\item[]
\item An NM-chain satisfies
\begin{equation}
\tag{$S_n(x_0,\dots,x_n)$}\bigwedge_{i<n}((x_i \to x_{i+1}) \to x_{i+1}) \to \bigvee_{i<n+1}x_i
\end{equation}
if and only if it has less than $2n + 2$ elements.
\item A nontrivial NM-chain satisfies
\begin{equation}\label{eq:fix}
\tag{$BP(x)$}\neg(\neg x^2)^2\leftrightarrow(\neg(\neg x)^2)^2
\end{equation}
if and only if it does not contain the negation fixpoint.
\end{enumerate}
\end{theorem}
\begin{remark}\label{rem:nmlog}	
As pointed out in \Cref{teo:log}, at the propositional level there is only one infinite-valued Gödel logic, that is every infinite G-chain generates the whole variety of G-algebras.

In the case of NM this is not true: indeed in \cite{gis} (see also \cite{ct}) it is shown that the variety generated by an infinite NM-chain without negation fixpoint corresponds to the one associated to the logic NM$^-$, i.e. NM plus \ref{eq:fix}. This result together with \Cref{teo:log} imply that there are (if we restrict to the logics associated to a totally ordered algebra) two different infinite-valued Nilpotent Minimum logic (NM and NM$^-$), at the propositional level.
\end{remark}
We now introduce a construction that allows to obtain an NM-chain starting from a G\"{o}del chain. This construction is an application of the ``connected rotation'' introduced in \cite{jen}.
\begin{definition}\label{def:rot}
Let $\mathcal{A}$ be a Gödel chain. We can construct an NM-chain $\mathcal{A}_\text{NM}$ in the following way:
\begin{itemize}
\item $A_\text{NM}=B \cup \{f\} \cup B'$, where $\lag B,\leq_{\mathcal{A}_\text{NM}}\rog=\lag A\setminus\{0\},\leq_{\mathcal{A}}\rog$ and $\lag B'=\{b': b\in B\},\leq_{\mathcal{A}_\text{NM}}\rog\simeq \lag B,\geq_{\mathcal{A}_\text{NM}}\rog$.
\item For every $x\in B, y\in B'$ set $x>_{\mathcal{A}_\text{NM}}f>_{\mathcal{A}_\text{NM}}y$.
\item Define a strong negation function $n:A_\text{NM}\to A_\text{NM}$ such that $n(f)=f$, $n(a)=a'$ and $n(b')=b$, for every $a\in B$ and $b'\in B'$.
\end{itemize}
It is easy to see that $\mathcal{A}_\text{NM}$ has negation fixpoint $f$, $B$ is the set of positive elements and $B'$ the set of negative elements: note that $1=\max(B)$ and $1'=\min(B')$ are the maximum and minimum of $\mathcal{A}_\text{NM}$. The element $1'$ will be called $0$.
\end{definition}
\begin{remark}\label{rem:nmcomp}
\begin{itemize}
\item An immediate consequence of \Cref{def:rot} is that $\lag A,\leq_\mathcal{A}\rog$ is order isomorphic to $\lag B\cup\{f\},\leq_{\mathcal{A}_\text{NM}}\rog$. From this fact it is easy to check that $\mathcal{A}$ is complete if and only if $\mathcal{A}_\text{NM}$ is.
\item It is an exercise to check that if $\mathcal{A}=G_\uparrow$, then $\mathcal{A}_\text{NM}=NM_\infty$, and if $\mathcal{A}=G_\downarrow$, then $\mathcal{A}_\text{NM}={NM'}_\infty$. Moreover $NM_\infty^-$, and ${NM'}_\infty^-$ are the subalgebras without negation fixpoint of, respectively, $NM_\infty^-$, and ${NM'}_\infty^-$.
\end{itemize}
\end{remark}

\subsection{First-order Nilpotent Minimum and Gödel Logics}
In this section we present the first-order versions of NM and G, called NM$\forall$ and G$\forall$: more details can be found in \cite{mtl,haj}.

A first-order language (we restrict to countable languages) is a pair $\lag \mathbf{P},\mathbf{C}\rog$, where $\mathbf{P}$ is a set of predicate symbols and $\mathbf{C}$ a set of constants (in general we do not need function symbols: see \cite{pred} for a development in this sense): we have the ``classical'' quantifiers $\forall,\exists$. The notions of term, formula, closed formula, term substitutable in a formula are defined like in the classical case (\cite{pred}, \cite{haj}), the connectives are those of the propositional level.

A theory is a set of closed formulas.
\paragraph*{}
Let L be NM or G or an axiomatic extension of them: then its first-order version, L$\forall$, is axiomatized as follows:
\begin{itemize}
\item The axioms resulting from the axioms of L by the substitution of the propositional variables by the first-order formulas
\item The following axioms\footnote{For the case of NM$\forall$, in \cite[theorems 2.31 and 2.32]{pred} it is showed that the axioms $(\exists 1)$ and $(\exists 2)$ are redundant: to maintain the notation of \cite{mtl} we give the full list.}:
\begin{align}
\tag{$\forall 1$}&(\forall x)\varphi(x)\rightarrow \varphi(x/t)\text{( }t\text{ substitutable for }x\text{ in }\varphi(x)\text{)}\\
\tag{$\exists 1$}&\varphi(x/t)\rightarrow (\exists x)\varphi(x)\text{( }t\text{ substitutable for }x\text{ in }\varphi(x)\text{)}\\
\tag{$\forall 2$}&(\forall x)(\nu \rightarrow \varphi)\rightarrow (\nu \rightarrow (\forall x)\varphi)\text{ (}x\text{ not free in }\nu\text{)}\\
\tag{$\exists 2$}&(\forall x)(\varphi \rightarrow \nu)\rightarrow ((\exists x)\varphi\rightarrow \nu)\text{ (}x\text{ not free in }\nu\text{)}\\
\tag{$\forall 3$}&(\forall x)(\varphi \vee \nu)\rightarrow ((\forall x)\varphi \vee \nu)\text{ (}x\text{ not free in }\nu\text{)}
\end{align}
\end{itemize}
The rules of L$\forall$ are: Modus Ponens: $\frac{\varphi\quad\varphi\to\psi}{\psi}$ and Generalization: $\frac{\varphi}{\forall x\varphi}$.
\paragraph*{}
As regards to semantics, we need to restrict to L-chains: given an L-chain $\mathbf{A}$, an $\mathbf{A}$-interpretation (or $\mathbf{A}$-model) is a structure $\mathbf{M}=\lag M,\{m_c\}_{c\in \mathbf{C}},\{r_P\}_{P\in \mathbf{P}}\rog$, where
\begin{itemize}
\item M is a non empty set.
\item for each $c\in \mathbf{C}$, $m_c\in M$
\item for each $P\in \mathbf{P}$ of arity\footnote{If $P$ has arity zero, then $r_P\in A$.} $n$, $r_P:M^n\to A$.
\end{itemize}
 For each evaluation over variables $v:Var\to M$, the truth value of a formula $\varphi$ ($\Vert\varphi\Vert_{\mathbf{M},v}^\mathbf{A}$) is defined inductively as follows:
\begin{itemize}
 \item $\Vert P(x,\dots,c,\dots)\Vert_{\mathbf{M},v}^\mathbf{A}=r_P(v(x),\dots,m_c,\dots)$
\item The truth value commutes with the connectives of L$\forall$, i.e.
\begin{align*}
\Vert\varphi\rightarrow\psi\Vert^\mathbf{A}_{\mathbf{M},v}&=\Vert\varphi\Vert^\mathbf{A}_{\mathbf{M},v}\Rightarrow\Vert\psi\Vert^\mathbf{A}_{\mathbf{M},v}\\
\Vert\varphi\&\psi\Vert^\mathbf{A}_{\mathbf{M},v}&=\Vert\varphi\Vert^\mathbf{A}_{\mathbf{M},v}*\Vert\psi\Vert^\mathbf{A}_{\mathbf{M},v}\\
\Vert \bot \Vert^\mathbf{A}_{\mathbf{M},v}=0&;\ \Vert \top \Vert^\mathbf{A}_{\mathbf{M},v}=1\\ \Vert\varphi\land\psi\Vert^\mathbf{A}_{\mathbf{M},v}&=\Vert\varphi\Vert^\mathbf{A}_{\mathbf{M},v}\sqcap\Vert\psi\Vert^\mathbf{A}_{\mathbf{M},v}
\end{align*}
\item $\Vert(\forall x)\varphi\Vert_{\mathbf{M},v}^\mathbf{A}=\inf\{\Vert\varphi\Vert_{M,v'}^\mathbf{A}:\ v'\equiv_x v$, i.e. $v'(y)=v(y)$ for all variables except for $x\}$
\item $\Vert(\exists x)\varphi\Vert_{\mathbf{M},v}^\mathbf{A}=\sup\{\Vert\varphi\Vert_{M,v'}^\mathbf{A}:\ v'\equiv_x v$, i.e. $v'(y)=v(y)$ for all variables except for $x\}$
\end{itemize}
if these $\inf$ and $\sup$ exist in $\mathbf{A}$, otherwise the truth value is undefined.

A model $\mathbf{M}$ is called $\mathbf{A}$-safe if all $\inf$ e $\sup$ necessary to define the truth value of each formula exist in $\mathbf{A}$. In this case, the truth value of a formula $\varphi$ over an $\mathbf{A}$-safe model is
\begin{equation*}
\Vert \varphi \Vert_\mathbf{M}^\mathbf{A}=\inf\{\Vert \varphi\Vert_{\mathbf{M},v}^\mathbf{A}:\ v:\ Var\to M\}
\end{equation*}
Note that if $\mathbf{A}$ is a standard algebra or has a lattice-reduct that is a complete lattice, then each $\mathbf{A}$-model is safe; obviously each finite $\mathbf{A}$-model ($M$ finite) is safe.

Finally, the notions of completeness are defined analogously to propositional level, with the difference that, with the notation $\models_\mathbf{A}\varphi$, we mean that $\Vert \varphi \Vert_{\mathbf{M}}^\mathbf{A}=1$, for each safe $\mathbf{A}$-interpretation $\mathbf{M}$.
\begin{theorem}[{\cite[Theorem 5.3.3]{haj},\cite[theorem 9]{mtl}}]\label{teo:stdcomp}
Let $\text{L}\in\{\text{NM,G}\}$. For each theory $T$ and formula $\varphi$ it holds that
\begin{equation*}
T\models_{L\forall}\varphi\qquad\text{iff}\qquad T\models_{[0,1]_\text{L}}\varphi.
\end{equation*}
\end{theorem}
\begin{remark}
Henceforth we will assume that the first-order language (of the type specified in the previous section) is fixed.
\end{remark}
\section{First-order Nilpotent Minimum logics}
In this section we study and compare the sets of (first-order) tautologies associated to four different infinite NM-chains (NM$_\infty$, NM${'}_\infty$, NM$_\infty^-$, NM${'}_\infty^-$), and to the finite ones. The choice of the four infinite chains, as explained in \Cref{rem:nmcomp}, is due to their relations with the Gödel chains $G_\uparrow$ and $G_\downarrow$: as we will see in \Cref{subsec:undec} this connection will help in the analysis of decidability problems.

\paragraph*{}Let $\mathcal{A}$ be an NM-chain: with the notation $TAUT_\mathcal{A}\forall$ we will denote the first-order tautologies of $\mathcal{A}$.
\begin{theorem}\label{teo:incstd}
For every NM-chain $\mathcal{A}$ it holds that $TAUT_{[0,1]_\text{NM}}\forall\subseteq TAUT_\mathcal{A}\forall$.
\end{theorem}
\begin{proof}
Immediate from \Cref{teo:stdcomp} and chain completeness theorem for NM$\forall$ (see \cite[theorem 7]{mtl}).
\end{proof}
A general result, concerning the subalgebras of $[0,1]_\text{NM}$, is the following.
\begin{proposition}\label{prop:inc}
Let $V,W$ be the universes of two subalgebras $\mathcal{V},\mathcal{W}$ of $[0,1]_\text{NM}$ (i.e. $V,W$ are two subsets of $[0,1]$ closed w.r.t. $n(x)=1-x$). If $V\subseteq W$, then $TAUT_\mathcal{W}\forall\subseteq TAUT_\mathcal{V}\forall$.
\end{proposition}
\begin{proof}
Let $\phi$ be the identity mapping over $W$, restricted to $V$: from the way in which the operations of an NM-chain are defined, an easy check shows that $\phi$ is a complete embedding (i.e. it preserves all $\inf$ and $\sup$) from $\mathcal{V}$ to $\mathcal{W}$. From this fact, if $\Vert\varphi\Vert_{\mathbf{M},v}^\mathcal{V}=\alpha<1$, then we can easily construct a model $\mathbf{M}'$ such that $\Vert\varphi\Vert_{\mathbf{M}',v}^\mathcal{W}=\alpha$.
\end{proof}
Now we analyze the differences between the (first-order) tautologies of $[0,1]_\text{NM}$ and those of the other four infinite chains that we have introduced.
\begin{theorem}\label{teo:nmopen}
\begin{enumerate}
\item $TAUT_{NM_\infty}\forall\subset TAUT_{NM^-_\infty}\forall$, $TAUT_{NM'_\infty}\forall\subset TAUT_{NM{'}^-_\infty}\forall$.
\item $TAUT_{[0,1]_\text{NM}}\forall\subset TAUT_{NM_\infty}\forall$, $TAUT_{[0,1]_\text{NM}}\forall\subset TAUT_{NM{'}^-_\infty}\forall$ and $TAUT_{NM_\infty}\forall\neq TAUT_{NM'_\infty}\forall$. 
\end{enumerate}
\end{theorem}
\begin{proof}
\begin{enumerate}
\item Immediate from \Cref{prop:inc} and \Cref{teo:fixpointeq}.
\item We show that the formula 
\begin{equation} 
\tag*{(*)}(\forall x)(\varphi(x)\&\nu)\leftrightarrow((\forall x)\varphi(x)\&\nu)
\label{eq:rigcont}
\end{equation}
where $x$ does not occurs freely in $\nu$, is a tautology for $NM_\infty$ and $NM{'}^-_\infty$, but it fails in $NM'_\infty$ (and hence, from \Cref{teo:incstd}, it fails in $[0,1]_\text{NM}$).
 
First of all we show that \ref{eq:rigcont} fails in $NM'_\infty$. Consider the formula $(\forall x)(P(x)\&p)\leftrightarrow((\forall x)P(x)\&p)$, where $p$ is a predicate of arity zero. Construct a model $\mathbf{M}$ (that is necessarily safe, since $NM'_\infty$ is complete) such that its domain $M$ is $(\frac{1}{2},1]\cap NM'_\infty$, $p$ is interpreted as $\frac{1}{2}$ and $r_P(m)=m$, for each $m\in M$. An easy check shows that $\Vert(\forall x)(P(x)\&p)\leftrightarrow((\forall x)P(x)\&p)\Vert_\mathbf{M}^{NM'_\infty}=\frac{1}{2}$ and hence $NM'_\infty\not\models \ref{eq:rigcont}$.

Now we show that $NM_\infty\models\ref{eq:rigcont}$. We have to check that, for each $W\subseteq NM_\infty$ (observe that NM$_\infty$ is a complete lattice) and $y\in NM_\infty$, it holds that $\inf_{w\in W}(w*y)=\inf(W)*y$. Note that, if $W$ has minimum $m$, then $\inf(W)*y=m*y=\inf_{w\in W}(w*y)$. Suppose then that $W$ has infimum but not minimum: an easy check shows that $\inf(W)=0$. In this last case we have that $\inf(W)*y=0=\inf_{w\in W}(w*1)\geq\inf_{w\in W}(w*y)$.

Finally we analyze $NM{'}^-_\infty$. We have to show that $\inf_{w\in W}\{w*x\}=\inf(W)*x$, for each $W\subseteq NM{'}^-_\infty$ and $x\in NM{'}^-_\infty$, when these $\inf$ exist. If $W$ has a minimum, say $m$, then $\inf_{w\in W}\{w*x\}=m*x=\inf(W)*x$; if $W$ does not have minimum, then it does not have $\inf$ and we are not interested to this case.
\end{enumerate}
\end{proof}
\begin{lemma}\label{lem:ord}
Let $\mathcal{A}$ be an NM-chain: an element $a$ does not have predecessor\footnote{An element $x\in A$ has a predecessor if there is an element $p\in A$ such that $p<x$ and there are no other elements between $p$ and $x$ (i.e. there is no $c\in A$ such that $p<c<x$). The notion of successor is defined dually.} (successor) if and only if $n(a)$ does not have successor (predecessor).
\end{lemma}
\begin{proof}
Immediate from the properties of the negation.
\end{proof}
Consider now the following formulas:
\begin{align}
\tag*{$C_\uparrow$}&(\exists x)(\varphi(x)\to\forall y \varphi(y))\label{form:cu}\\
\tag*{$C_\downarrow$}&(\exists x)(\exists y \varphi(y)\to \varphi(x)).\label{form:cd}
\end{align}
Their names are due to the fact that, in the context of Gödel logics they are equivalent to ask, respectively, over a Gödel-chain that ``every infimum is a minimum'', and ``every supremum is a maximum'': see \cite{fgod} for details.
\begin{theorem}
The formulas \ref{form:cu} and \ref{form:cd} hold in every NM-chain $\mathcal{A}$ in which every element of $\mathcal{A}\setminus\{0,1\}$ has a predecessor in $\mathcal{A}$. They both fail in any other NM-chain.
\end{theorem}
\begin{proof}
Let $\mathcal{B}$ be an NM-chain that has an element $x\in\mathcal{B}\setminus\{0,1\}$ without predecessor in $\mathcal{B}$.

Consider the set $W=\{w\in\mathcal{B}:\ w<x\}$: direct inspection shows that $\sup_{w\in W}\{\sup(W)\Rightarrow w\}=\sup_{w\in W}\{x\Rightarrow w\}=\sup_{w\in W}\{\max(n(x),w)\}<1$. This shows that $\mathcal{B}\not\models$\ref{form:cd}.

From \Cref{lem:ord} we know that $n(x)$ does not have successor. Construct the set $W=\{w\in\mathcal{B}:\ w>n(x)\}$: direct inspection shows that $\sup_{w\in W}\{w\Rightarrow\inf(W)\}=\sup_{w\in W}\{w\Rightarrow n(x)\}=\sup_{w\in W}\{\max(n(w),n(x))\}<1$. This shows that $\mathcal{B}\not\models$ \ref{form:cu}.

\smallskip
Let $\mathcal{A}$ be any NM-chain in which every element of $\mathcal{A}\setminus\{0,1\}$ has a predecessor in $\mathcal{A}$: it follows that every element of $\mathcal{A}\setminus\{0,1\}$ has predecessor and successor in $\mathcal{A}$. We have to check that $\sup_{w\in W}\{w\Rightarrow\inf(W)\}=1$ and $\sup_{w\in W}\{\sup(W)\Rightarrow w\}=1$, for every $W$ in which these $\inf$ and $\sup$ exist. If $W$ has minimum $m$, then $\sup_{w\in W}\{w\Rightarrow\inf(W)\}=m\Rightarrow m=1$; if $W$ has maximum $n$, then $\sup_{w\in W}\{\sup(W)\Rightarrow w\}=n\Rightarrow n=1$. If $W$ has infimum, but not minimum, then $\inf(W)=0$ and $\sup_{w\in W}\{w\Rightarrow\inf(W)\}=\sup_{w\in W}\{n(w)\}=1$. Finally, if $W$ has supremum, but not maximum, then $\sup(W)=1$ and $\sup_{w\in W}\{\sup(W)\Rightarrow w\}=\sup_{w\in W}\{1\Rightarrow w\}=1$.
\end{proof}
\begin{corollary}
\begin{itemize}
\item[]
\item \ref{form:cd} and \ref{form:cu} belong to $TAUT_{NM_\infty}\forall$, $TAUT_{NM^-_\infty}\forall$, $TAUT_{NM{'}^-_\infty}\forall$ and $TAUT_{NM_n}\forall$, for every $1<n<\omega$.
\item \ref{form:cd} and \ref{form:cu} fail in $[0,1]_\text{NM}$ and $NM'_\infty$.
\end{itemize}
\end{corollary}
\begin{remark}
Continuing with the analogies with Gödel logic, it can be showed (see \cite{fgod} and \cite{tgod}) that \ref{form:cd} and \ref{form:cu} are tautologies in $G_\uparrow$ and in every finite Gödel chain, whilst $G_\downarrow\not\models$\ref{form:cu} and $G_\downarrow\models$\ref{form:cd}. Both the formulas fail in $G\forall$ (see \cite{tgod}).
\end{remark}
We prosecute our analysis of first-order tautologies with the following
\begin{theorem}\label{teo:nminf}
Let $\varphi$ be an NM$\forall$ formula. For every integer $n>1$ and every \emph{even} integer $m>1$ it holds that
\begin{itemize}
\item If $NM_n\not\models\varphi$, then $NM_\infty\not\models\varphi$ and $NM'_\infty\not\models\varphi$.
\item If $NM_m\not\models\varphi$, then $NM^-_\infty\not\models\varphi$ and $NM{'}^-_\infty\not\models\varphi$.
\end{itemize}
\end{theorem}
\begin{proof}
It is enough to show that $NM_n\hookrightarrow NM_\infty$, $NM_n\hookrightarrow NM'_\infty$, $NM_m\hookrightarrow NM^-_\infty$, $NM_m\hookrightarrow NM{'}^-_\infty$ preserving all $\inf$ and $\sup$.

We begin with the case of $NM_\infty$.

Let $0=c_1<c_2<\dots<c_n=1$ be the elements of NM$_n$: consider a map $\phi$ such that
\begin{itemize}
\item $\phi(c_1)=0$ and $\phi(c_n)=1$.
\item If $NM_n$ has a fixpoint $f$, then $\phi(f)=\frac{1}{2}$.
\item Let $c_k$ be the least positive element: we set $\phi(c_j)=1-\frac{1}{3+(j-k)}$ for every $c_n>c_j\geq c_k$.
\item Let $c_h$ be the greatest negative element: we set $\phi(c_i)=\frac{1}{3+(h-i)}$ for every $c_1<c_i\leq c_h$.
\end{itemize}
A direct inspection shows that $\phi$ is an embedding from the two chains. Moreover, since $NM_n$ is finite, then for each $W\subseteq NM_n$, $\phi(\inf(W))=\phi(\min(W))=\min(\phi(W))$; analogously for $\sup$.

Concerning the case of $NM'_\infty$ we have only to modify the map $\phi$ and the proof proceeds analogously to the previous case.

Let $0=c_1<c_2<\dots<c_n=1$ be the elements of NM$_n$: consider a map $\phi^\prime$ such that
\begin{itemize}
\item $\phi^\prime(c_1)=0$ and $\phi^\prime(c_n)=1$.
\item If $NM_n$ has a fixpoint $f$, then $\phi^\prime(f)=\frac{1}{2}$.
\item Let $c_k$ be the greatest positive element of $NM'_\infty\setminus\{1\}$: we set $\phi^\prime(c_j)=\frac{1}{2}+\frac{1}{2(2+k-j)}$ for every $c_n>c_k\geq c_j$.
\item Let $c_h$ be the least negative element of $NM'_\infty\setminus\{0\}$: we set $\phi^\prime(c_i)=\frac{1}{2}-\frac{1}{2(2+i-h)}$ for every $c_1<c_h\leq c_i$.
\end{itemize}
Finally the proofs for $NM^-_\infty$ and $NM{'}^-_\infty$ are identical to the previous ones, except for the absence of the negation fixpoint.
\end{proof}
\begin{corollary}\label{cor:inc}
For every integer $n>1$ we have $TAUT_{[0,1]_\text{NM}}\forall\subset TAUT_{NM_n}\forall$, $TAUT_{NM_\infty}\forall\subset TAUT_{NM_n}\forall$,  $TAUT_{NM'_\infty}\forall\subset TAUT_{NM_n}\forall$. Moreover, if $n$ is even, then $TAUT_{NM^-_\infty}\forall\subset TAUT_{NM_n}\forall$, $TAUT_{NM{'}^-_\infty}\forall\subset TAUT_{NM_n}\forall$.
\end{corollary}
\begin{proof}
From \Cref{teo:incstd,teo:nminf} we have the non-strict inclusions. To prove the strictness, direct inspection shows that the formula $\bigvee_{0<i<n}(p_i\to p_{i+1})$ (where each $p_i$ is a predicate of arity zero) is a first-order tautology of NM$_n$, but it fails in every infinite NM-chain.
\end{proof}
Differently from the results of \cite{fgod} for G$_n$, it cannot be showed that $TAUT_{NM_{n+1}}\forall\subset TAUT_{NM_n}\forall$. Indeed, if NM$_n$ has negation fixpoint, then (see \Cref{teo:fixpointeq}) $NM_n\not\models\neg(\neg p^2)^2\leftrightarrow(\neg(\neg p)^2)^2$, where $p$ is a predicate of arity zero. However $NM_{n+1}\models\neg(\neg p^2)^2\leftrightarrow(\neg(\neg p)^2)^2$.

Note that if both the chains are with (without) negation fixpoint, then the previous problem disappear; note also that NM$_n$ has negation fixpoint if and only if $n$ is odd.
\paragraph*{}
Hence, we have the following
\begin{theorem}
For each pair of integers $m,n$ such that $1<m<n$, if $m,n$ are both even (odd), then $TAUT_{NM_{n}}\forall\subset TAUT_{NM_m}\forall$.
\end{theorem}
\begin{proof}
It is enough to check that $NM_m\hookrightarrow NM_n$, preserving all $\inf$ and $\sup$. Take an injective map $\phi$ from the lattice reduct of $NM_m$ to the one of $NM_n$ such that:
\begin{itemize}
\item $\phi(0)=0$, $\phi(1)=1$.
\item if  $NM_m, NM_n$ have negation fixpoints $f, f'$, then $\phi(f)=f'$.
\item $\phi$ maps all the positive elements of $NM_m$ into the ones of $NM_n$, preserving the order. That is, for every $x,y\in NM_m^+$ with $x<y$ it holds that $\phi(x), \phi(y)\in NM_n^+$ and $\phi(x)<\phi(y)$.
\item for every negative element $x\in NM_m^-$, $\phi(x)=1-\phi(1-x)$.
\end{itemize}
An easy check shows that $\phi$ is an embedding that preserves all $\inf$ and $\sup$. This shows that $TAUT_{NM_{n}}\forall\subseteq TAUT_{NM_m}\forall$.

To conclude, note that $NM_m\models\bigvee_{0<i<m}(p_i\to p_{i+1})$, but $NM_n\not\models\bigvee_{0<i<m}(p_i\to p_{i+1})$: hence $TAUT_{NM_{n}}\forall\subset TAUT_{NM_m}\forall$.
\end{proof}
Moreover, by inspecting the previous proof, we obtain
\begin{corollary}
For every \emph{even} integer $n>1$,  it holds that $TAUT_{NM_{n+1}}\forall\subset TAUT_{NM_n}\forall$.
\end{corollary}
In \cite{fgod} it is shown that the first-order tautologies of G$_\uparrow$ are the first-order formulas valid in all finite Gödel-chains. We will show that, under this point of view, NM$_\infty$ plays the same role of G$_\uparrow$: that is $TAUT_{NM_\infty}\forall=\bigcap_{n\geq 2}TAUT_{NM_n}\forall$. 

We start with the following lemma, that says that if an $NM_\infty$-model assigns to the atomic formulas truth values between a value $\alpha$ and its negation, then the same holds for every other formula.
\begin{lemma}\label{lem:modfin}
Let $\mathbf{M}=\lag M,\{r_p\}_{p\in\mathbf{P}},\{m_c\}_{c\in\mathbf{C}}\rog$ be an $NM_\infty$-model. For $\alpha\in NM_\infty$, consider the $NM_\infty$-model $\mathbf{M}_\alpha=\lag M,\{r'_p\}_{p\in\mathbf{P}},\{m_c\}_{c\in\mathbf{C}}\rog$ such that, for each atomic formula $\psi$ and every evaluation $v$
\begin{equation}
\tag{m}\Vert\psi\Vert_{\mathbf{M}_\alpha,v}=\begin{cases}
1 &\text{if }\Vert\psi\Vert_{\mathbf{M},v}>|\alpha|\\
0 &\text{if }\Vert\psi\Vert_{\mathbf{M},v}<n(|\alpha|)\\
\Vert\psi\Vert_{\mathbf{M},v}&\text{otherwise}
\end{cases}
\label{eq:modfin}
\end{equation}
Where $|\alpha|=\max(\alpha, n(\alpha))$.

Then (\ref{eq:modfin}) holds for every first-order formula $\varphi$.
\end{lemma}
\begin{proof}
By structural induction. Since $\mathbf{M}_\alpha$ and $\mathbf{M}_{n(\alpha)}$ define the same model we will assume, without loss of generality, that $\alpha\geq\frac{1}{2}$ (otherwise we set $\alpha=n(\alpha)$).
\begin{itemize}
\item If $\varphi$ is atomic or $\bot$, then there is nothing to prove.
\item $\varphi\df\psi\land\chi$ and the claim holds for $\psi$ and $\chi$. First of all note that $\Vert\psi\land\chi\Vert_{\mathbf{M},v}=\min(\Vert\psi\Vert_{\mathbf{M},v},\Vert\chi\Vert_{\mathbf{M},v})$ and $\Vert\psi\land\chi\Vert_{\mathbf{M}_\alpha,v}=\min(\Vert\psi\Vert_{\mathbf{M}_\alpha,v},\Vert\chi\Vert_{\mathbf{M}_\alpha,v})$: from the induction hypothesis, if $\Vert\psi\Vert_{\mathbf{M},v}=\Vert\chi\Vert_{\mathbf{M},v}$, then the lemma holds.

For the other cases, note that if $\Vert\psi\Vert_{\mathbf{M},v}<\Vert\chi\Vert_{\mathbf{M},v}$ ($>$), then $\Vert\psi\Vert_{\mathbf{M}_\alpha,v}\leq\Vert\chi\Vert_{\mathbf{M}_\alpha,v}$ ($\geq$). Suppose that $\Vert\psi\Vert_{\mathbf{M},v}<\Vert\chi\Vert_{\mathbf{M},v}$. If $\Vert\psi\Vert_{\mathbf{M}_\alpha,v}<\Vert\chi\Vert_{\mathbf{M}_\alpha,v}$  then, applying the induction hypothesis, we have the result. The other case is $\Vert\psi\Vert_{\mathbf{M}_\alpha,v}=\Vert\chi\Vert_{\mathbf{M}_\alpha,v}\in\{0,1\}$: clearly either $\Vert\chi\Vert_{\mathbf{M},v}<n(\alpha)$ or $\Vert\psi\Vert_{\mathbf{M},v}>\alpha$. Again, applying the induction hypothesis, the claim follows.
\item $\varphi\df\psi\&\chi$ and the claim holds for $\psi$ and $\chi$. We have two cases:
\begin{itemize}
\item $\Vert\varphi\Vert_{\mathbf{M},v}=0$: this happens if and only if $\Vert\psi\Vert_{\mathbf{M},v}\leq n(\Vert\chi\Vert_{\mathbf{M},v})$. Direct inspection shows that this implies $\Vert\psi\Vert_{\mathbf{M}_\alpha,v}\leq n(\Vert\chi\Vert_{\mathbf{M}_\alpha,v})$ and hence $\Vert\varphi\Vert_{\mathbf{M}_\alpha,v}=0$.
\item $\Vert\varphi\Vert_{\mathbf{M},v}=\min(\Vert\psi\Vert_{\mathbf{M},v},\Vert\chi\Vert_{\mathbf{M},v})>0$: this happens if and only if $\Vert\psi\Vert_{\mathbf{M},v}> n(\Vert\chi\Vert_{\mathbf{M},v})$.

    If $\Vert\psi\Vert_{\mathbf{M},v}<n(\alpha)$ then $\Vert\varphi\Vert_{\mathbf{M},v}<n(\alpha)$ and $\Vert\psi\Vert_{\mathbf{M}_\alpha,v}=0=\Vert\varphi\Vert_{\mathbf{M}_\alpha,v}$.

    If $n(\alpha)\leq\Vert\psi\Vert_{\mathbf{M},v}\leq\alpha$, then $\Vert\psi\Vert_{\mathbf{M},v}=\Vert\psi\Vert_{\mathbf{M}_\alpha,v}$ and $n(\Vert\chi\Vert_{\mathbf{M},v})<\alpha$: if $n(\alpha)\leq n(\Vert\chi\Vert_{\mathbf{M}_\alpha,v})$, then $\Vert\varphi\Vert_{\mathbf{M},v}=\Vert\varphi\Vert_{\mathbf{M}_\alpha,v}$, otherwise $n(\Vert\chi\Vert_{\mathbf{M},v})<n(\alpha)$, $\Vert\chi\Vert_{\mathbf{M},v}>\alpha$ and hence $\Vert\varphi\Vert_{\mathbf{M},v}=\Vert\psi\Vert_{\mathbf{M},v}=\Vert\varphi\Vert_{\mathbf{M}_\alpha,v}$, since $\Vert\chi\Vert_{\mathbf{M}_\alpha,v}=1$, due to the induction hypothesis.

    Finally, suppose that $\Vert\psi\Vert_{\mathbf{M},v}>\alpha$. We have that $\Vert\psi\Vert_{\mathbf{M}_\alpha,v}=1$: if $n(\Vert\chi\Vert_{\mathbf{M},v})>\alpha$, then $\Vert\chi\Vert_{\mathbf{M},v}<n(\alpha)$ and hence $\Vert\varphi\Vert_{\mathbf{M},v}=\Vert\chi\Vert_{\mathbf{M},v}$, from which we have $\Vert\chi\Vert_{\mathbf{M}_\alpha,v}=0=\Vert\varphi\Vert_{\mathbf{M}_\alpha,v}$. If $n(\alpha)\leq n(\Vert\chi\Vert_{\mathbf{M},v})\leq\alpha$, then the same holds for $\Vert\chi\Vert_{\mathbf{M},v}$ and we have $\Vert\varphi\Vert_{\mathbf{M},v}=\Vert\chi\Vert_{\mathbf{M},v}=\Vert\chi\Vert_{\mathbf{M}_\alpha,v}=\Vert\varphi\Vert_{\mathbf{M}_\alpha,v}$. If $n(\Vert\chi\Vert_{\mathbf{M},v})<n(\alpha)$, then $\Vert\chi\Vert_{\mathbf{M},v}>\alpha$ and hence $\Vert\chi\Vert_{\mathbf{M}_\alpha,v}=\Vert\psi\Vert_{\mathbf{M}_\alpha,v}=1=\Vert\varphi\Vert_{\mathbf{M}_\alpha,v}$.
\end{itemize}
\item $\varphi\df\psi\to\chi$ and the claim holds for $\psi$ and $\chi$. We have two cases.
\begin{itemize}
\item $\Vert\psi\Vert_{\mathbf{M},v}\leq\Vert\chi\Vert_{\mathbf{M},v}$: as we have already noticed, this implies $\Vert\psi\Vert_{\mathbf{M}_\alpha,v}\leq\Vert\chi\Vert_{\mathbf{M}_\alpha,v}$ and we have that $\Vert\varphi\Vert_{\mathbf{M},v}=1=\Vert\varphi\Vert_{\mathbf{M}_\alpha,v}$.
\item $\Vert\psi\Vert_{\mathbf{M},v}>\Vert\chi\Vert_{\mathbf{M},v}$: it is not difficult to check that $\Vert\psi\Vert_{\mathbf{M}_\alpha,v}\geq\Vert\chi\Vert_{\mathbf{M}_\alpha,v}$.

    If the equality holds, then $\Vert\psi\Vert_{\mathbf{M}_\alpha,v}=\Vert\chi\Vert_{\mathbf{M}_\alpha,v}\in\{0,1\}$ and either $\Vert\chi\Vert_{\mathbf{M},v}>\alpha$ or $\Vert\psi\Vert_{\mathbf{M},v}<n(\alpha)$: in both the cases $\Vert\varphi\Vert_{\mathbf{M},v}=\max(n(\Vert\psi\Vert_{\mathbf{M},v}),\Vert\chi\Vert_{\mathbf{M},v})$. If $\Vert\chi\Vert_{\mathbf{M},v}>\alpha$, then $n(\Vert\psi\Vert_{\mathbf{M},v})<n(\alpha)$ and $\Vert\varphi\Vert_{\mathbf{M},v}=\Vert\chi\Vert_{\mathbf{M},v}>\alpha$: from these facts we have $\Vert\psi\Vert_{\mathbf{M}_\alpha,v}=\Vert\chi\Vert_{\mathbf{M}_\alpha,v}=1=\Vert\varphi\Vert_{\mathbf{M}_\alpha,v}$. If $\Vert\psi\Vert_{\mathbf{M},v}<n(\alpha)$, then $n(\Vert\psi\Vert_{\mathbf{M},v}),\Vert\varphi\Vert_{\mathbf{M},v}>\alpha$ and from the induction hypothesis we have $\Vert\psi\Vert_{\mathbf{M}_\alpha,v}=0=\Vert\chi\Vert_{\mathbf{M}_\alpha,v}$ and $\Vert\varphi\Vert_{\mathbf{M}_\alpha,v}=1$.

    The last case is $\Vert\psi\Vert_{\mathbf{M}_\alpha,v}>\Vert\chi\Vert_{\mathbf{M}_\alpha,v}$: we have that $\Vert\varphi\Vert_{\mathbf{M},v}=\max(n(\Vert\psi\Vert_{\mathbf{M},v}),\Vert\chi\Vert_{\mathbf{M},v})$ and $\Vert\varphi\Vert_{\mathbf{M}_\alpha,v}=\max(n(\Vert\psi\Vert_{\mathbf{M}_\alpha,v}),\Vert\chi\Vert_{\mathbf{M}_\alpha,v})$.

There are two subcases.

\subitem $n(\Vert\psi\Vert_{\mathbf{M},v})>\Vert\chi\Vert_{\mathbf{M},v}$: clearly $\Vert\varphi\Vert_{\mathbf{M},v}=n(\Vert\psi\Vert_{\mathbf{M},v})$. If $n(\alpha)\leq\Vert\psi\Vert_{\mathbf{M},v}\leq\alpha$, then we have that $\Vert\psi\Vert_{\mathbf{M},v}=\Vert\psi\Vert_{\mathbf{M}_\alpha,v}$, $n(\Vert\psi\Vert_{\mathbf{M},v})=n(\Vert\psi\Vert_{\mathbf{M}_\alpha,v})$ and $\Vert\varphi\Vert_{\mathbf{M}_\alpha,v}=\Vert\varphi\Vert_{\mathbf{M},v}=n(\Vert\psi\Vert_{\mathbf{M},v})$ (noting that $\Vert\chi\Vert_{\mathbf{M}_\alpha,v}\leq \Vert\chi\Vert_{\mathbf{M},v}$, since $\Vert\psi\Vert_{\mathbf{M}_\alpha,v}>\Vert\chi\Vert_{\mathbf{M}_\alpha,v}$). If $\Vert\psi\Vert_{\mathbf{M},v}>\alpha$, then $\Vert\psi\Vert_{\mathbf{M}_\alpha,v}=1$, $n(\Vert\psi\Vert_{\mathbf{M},v})<n(\alpha)$ and $n(\Vert\psi\Vert_{\mathbf{M}_\alpha,v})=0$: from these facts and the hypothesis we obtain $n(\alpha)>\Vert\varphi\Vert_{\mathbf{M},v}=n(\Vert\psi\Vert_{\mathbf{M},v})>\Vert\chi\Vert_{\mathbf{M},v}$ and hence $\Vert\varphi\Vert_{\mathbf{M}_\alpha,v}=0=n(\Vert\psi\Vert_{\mathbf{M}_\alpha,v})=\Vert\chi\Vert_{\mathbf{M}_\alpha,v}$. The last case is $\Vert\psi\Vert_{\mathbf{M},v}<n(\alpha)$: we have that $\Vert\varphi\Vert_{\mathbf{M},v}=n(\Vert\psi\Vert_{\mathbf{M},v})>\alpha$ and hence $1=n(\Vert\psi\Vert_{\mathbf{M}_\alpha,v})=\Vert\varphi\Vert_{\mathbf{M}_\alpha,v}$.
\subitem $\Vert\chi\Vert_{\mathbf{M},v}>n(\Vert\psi\Vert_{\mathbf{M},v})$: we proceed analogously with the previous case.
\end{itemize}
\item $\varphi\df(\forall x)\psi(x)$ and the claim holds for $\psi(x)$: this means that, from the induction hypothesis, for every $v'\equiv_x v$ we have that (\ref{eq:modfin}) holds for $\Vert\psi(x)\Vert_{\mathbf{M},v'}$ and $\Vert\psi(x)\Vert_{\mathbf{M}_\alpha,v'}$.

We distinguish three cases.
\begin{itemize}
\item $\Vert(\forall x)\psi(x)\Vert_{\mathbf{M},v}<n(\alpha)$. Clearly there exists a $v'\equiv_x v$ such that $\Vert\psi(x)\Vert_{\mathbf{M},v'}<n(\alpha)$ and hence, applying the induction hypothesis, we have $\Vert\psi(x)\Vert_{\mathbf{M}_\alpha,v'}=0=\Vert(\forall x)\psi(x)\Vert_{\mathbf{M}_\alpha,v}$.
\item $\Vert(\forall x)\psi(x)\Vert_{\mathbf{M},v}>\alpha$. Clearly for each $v'\equiv_x v$ it holds that $\Vert\psi(x)\Vert_{\mathbf{M},v'}>\alpha$, $\Vert\psi(x)\Vert_{\mathbf{M}_\alpha,v'}=1$ (due to the induction hypothesis) and hence we have $\Vert(\forall x)\psi(x)\Vert_{\mathbf{M}_\alpha,v}=1$.
\item $n(\alpha)\leq\Vert(\forall x)\psi(x)\Vert_{\mathbf{M},v}\leq\alpha$. We have that $\Vert\psi(x)\Vert_{\mathbf{M},v'}\geq n(\alpha)$ for every $v'\equiv_x v$. Moreover there is at least a $v''\equiv_x v$ such that $\Vert\psi(x)\Vert_{\mathbf{M},v''}\leq\alpha$: due to the induction hypothesis for every such $v''$ we have that $\Vert\psi(x)\Vert_{\mathbf{M}_\alpha,v''}=\Vert\psi(x)\Vert_{\mathbf{M},v''}$. Applying again the induction hypothesis we have that $\Vert(\forall x)\psi(x)\Vert_{\mathbf{M}_\alpha,v}=\Vert(\forall x)\psi(x)\Vert_{\mathbf{M},v}$.
\end{itemize}
\end{itemize}
We do not analyze the case $\varphi\df(\exists x)\psi(x)$, since the two quantifiers are inter-definable, in NM$\forall$, as in classical logic (see \cite[theorem 2.31]{pred}).
\end{proof}
\begin{remark}\label{rem:stdcut}
It is not difficult to see that the previous lemma holds even for $[0,1]_\text{NM}$, using the same proof. 
\end{remark}
\begin{theorem}\label{teo:tautinc}
$TAUT_{NM_\infty}\forall=\bigcap_{n\geq 2} TAUT_{NM_n}\forall$.
\end{theorem}
\begin{proof}
The fact that $TAUT_{NM_\infty}\forall\subseteq\bigcap_{n\geq 2}TAUT_{NM_n}\forall$ follows from \Cref{cor:inc}.

Concerning the reverse inclusion, suppose that $\Vert\varphi\Vert^{NM_\infty}_{\mathbf{M},v}=\alpha<1$. Take $\alpha<\beta<1$: due to \Cref{lem:modfin} it is easy to check that $\Vert\varphi\Vert^{NM_\infty}_{\mathbf{M}_\beta,v}\leq\alpha$. Since $\mathbf{M}_\beta$ uses only a finite number of truth values, it is easy to construct a model $\mathbf{M}'_\beta$ (starting from $\mathbf{M}_\beta$ and modifying the range of the various $r'_P$'s) over an appropriate $NM_k$ such that $\Vert\varphi\Vert^{NM_k}_{\mathbf{M}'_\beta,v}=\Vert\varphi\Vert^{NM_\infty}_{\mathbf{M}_\beta,v}$.

\end{proof}
We now introduce a family of NM-chains that will be useful to give an equivalent characterization of $TAUT_{[0,1]_\text{NM}}\forall$.
\paragraph*{}
For $\alpha\in(0,1)$, let $\mathcal{A}_\alpha$ be the NM-chain defined over the universe $[1-|\alpha|,|\alpha|]\cup\{0,1\}$ and $n(x)=1-x$ (recall that $|\alpha|=\max(|\alpha|,n(|\alpha|))$): observe that $\mathcal{A}_\alpha$ and $\mathcal{A}_{n(\alpha)}$ are isomorphic and every chain of this type forms a complete lattice.
\paragraph*{}
Due to \Cref{rem:stdcut} and \Cref{teo:incstd}, with a proof very similar to the one of \Cref{teo:tautinc}, we obtain the following result: this is - mutatis mutandis - the analogous of \Cref{teo:tautinc} for $[0,1]_\text{NM}$ and the family of NM-chains previously introduced.
\begin{theorem}\label{teo:tautincstd}
$TAUT_{[0,1]_\text{NM}}\forall=\bigcap_{\alpha\in (0,1)} TAUT_{\mathcal{A}_\alpha}\forall$.
\end{theorem}
In classical (first-order) logic every formula can be written in prenex normal form: this is because the so called ``quantifiers shifting laws'' hold. 
For Nilpotent Minimum logic the situation is different: indeed, as shown in \Cref{teo:nmopen} some quantifier shifting laws fail in NM$\forall$. One can ask which is the situation for the NM-chains, about these formulas. 

The following theorem shows a characterization of the validity of these shifting laws in terms of the order type of an NM-chain.
\begin{theorem}\label{teo:prenex}
Consider the following formulas:
\begin{align}
&(\forall x)(\varphi(x)\land\nu)\leftrightarrow ((\forall x)\varphi(x)\land\nu)\\
&(\exists x)(\varphi(x)\land\nu)\leftrightarrow ((\exists x)\varphi(x)\land\nu)\\
&(\forall x)(\varphi(x)\vee\nu)\leftrightarrow ((\forall x)\varphi(x)\vee\nu)\\
&(\exists x)(\varphi(x)\vee\nu)\leftrightarrow ((\exists x)\varphi(x)\vee\nu)\\
&(\forall x)(\varphi(x)\land\psi(x))\leftrightarrow ((\forall x)\varphi(x)\land(\forall x)\psi(x))\\
&(\exists x)(\varphi(x)\land\psi(x))\leftrightarrow ((\exists x)\varphi(x)\land(\exists x)\psi(x))\\
&(\forall x)(\varphi(x)\vee\psi(x))\leftrightarrow ((\forall x)\varphi(x)\vee(\forall x)\psi(x))\\
&(\exists x)(\varphi(x)\vee\psi(x))\leftrightarrow ((\exists x)\varphi(x)\vee(\exists x)\psi(x))\\
&(\exists x)(\varphi(x)\&\nu)\leftrightarrow ((\exists x)\varphi(x)\&\nu)\\
&(\exists x)(\varphi(x)\&\psi(x))\leftrightarrow ((\exists x)\varphi(x)\&(\exists x)\psi(x))\\
&(\forall x)(\varphi(x)\to\nu)\leftrightarrow ((\exists x)\varphi(x)\to\nu)\\
&(\forall x)(\nu\to\varphi(x))\leftrightarrow (\nu\to(\forall x)\varphi(x))\\
&\neg(\exists x)\varphi(x)\leftrightarrow(\forall x)\neg\varphi(x)\\
&\neg(\forall x)\varphi(x)\leftrightarrow(\exists x)\neg\varphi(x)\\
&(\forall x)(\varphi(x)\&\nu)\leftrightarrow ((\forall x)\varphi(x)\&\nu)\\
&(\forall x)(\varphi(x)\&\psi(x))\leftrightarrow ((\forall x)\varphi(x)\&(\forall x)\psi(x))\\
&(\exists x)(\varphi(x)\to\nu)\leftrightarrow ((\forall x)\varphi(x)\to\nu)\\
&(\exists x)(\nu\to\varphi(x))\leftrightarrow (\nu\to(\exists x)\varphi(x))
\end{align}
where $x$ does not occurs freely in $\nu$. We have that
\begin{itemize}
\item The formulas (1)-(14) hold in every NM-chain.
\item The formulas (15)-(18) hold in every NM-chain $\mathcal{A}$ in which every element of $\mathcal{A}\setminus\{0,1\}$ has a predecessor in $\mathcal{A}$, and fail to hold in any other NM-chain.
\end{itemize}
\end{theorem}
\begin{proof}
A direct inspection.
\end{proof}
\begin{corollary}
\begin{itemize}
\item[]
\item The formulas (1)-(18) belong to $TAUT_{NM_\infty}\forall$, $TAUT_{NM^-_\infty}\forall$, $TAUT_{NM{'}^-_\infty}\forall$ and $TAUT_{NM_n}\forall$, for every $1<n<\omega$.
\item The formulas (1)-(14) belong to $TAUT_{[0,1]_\text{NM}}\forall$ and $TAUT_{NM'_\infty}\forall$.
\item The formulas (15)-(18) fail in $[0,1]_\text{NM}$ and $NM'_\infty$.
\end{itemize}
\end{corollary}
Finally, we summarize relationship (in terms of reciprocal inclusion) between the sets of tautologies of the NM-chains studied.
\begin{theorem}
For every integer $n>1$ and every $\emph{even}$ integer $m>1$
\begin{enumerate}
\item $TAUT_{[0,1]_\text{NM}}\forall=\bigcap_{\alpha\in (0,1)} TAUT_{\mathcal{A}_\alpha}\forall$.
\item $TAUT_{[0,1]_\text{NM}}\forall\subset TAUT_{NM_\infty}\forall\subset TAUT_{NM_n}\forall$.
\item $TAUT_{NM_\infty}\forall\subset TAUT_{NM^-_\infty}\forall\subset TAUT_{NM_m}\forall$, $TAUT_{NM'_\infty}\forall\subset TAUT_{NM{'}^-_\infty}\forall\subset TAUT_{NM_m}\forall$.
\item $TAUT_{[0,1]_\text{NM}}\forall\subseteq TAUT_{NM'_\infty}\forall\subset TAUT_{NM_n}\forall$.
\item $TAUT_{NM'_\infty}\forall\neq TAUT_{NM_\infty}\forall=\bigcap_{n\geq 2} TAUT_{NM_n}\forall$ and hence $TAUT_{NM'_\infty}\forall\subset TAUT_{NM_\infty}\forall$.
\end{enumerate}
\end{theorem}
This theorem can be improved: indeed in the next section we will show that $TAUT_{NM'_\infty}\forall$ is not recursively enumerable. As a consequence, we have that $TAUT_{[0,1]_\text{NM}}\forall\subset TAUT_{NM'_\infty}\forall$.

\subsection{Axiomatizability and undecidability}\label{subsec:undec}
In this section we study if the sets of first-order tautologies associated to the NM-chains till introduced are axiomatizable or not: that is, we investigate if, given one of the previous NM-chains, there is a logic that is complete w.r.t. it. As we will see, it will be the case only for finite NM-chains: for the other chains we will have undecidability results (the set of first-order tautologies will be not recursively axiomatizable) and one open problem.

From \cite[Theorem 3]{gis} we can state
\begin{theorem}\label{teo:lognmfin}
For every integer $n\geq 1$
\begin{itemize}
\item Let LNM$_{2n}$ be the logic obtained from NM with the axioms $S_n(x_0,\dots,x_n)$ and $BP(x)$ (see \Cref{teo:fixpointeq}). Then LNM$_{2n}$ is complete w.r.t. NM$_{2n}$. 
\item Let LNM$_{2n+1}$ be the logic obtained from NM with the axiom $S_n(x_0,\dots,x_n)$. Then LNM$_{2n+1}$ is complete w.r.t. NM$_{2n+1}$.
\end{itemize}
\end{theorem}
As regards to the first-order version of these logics, we have
\begin{theorem}
For each integer $n>1$ and each $NM\forall$ formula $\varphi$,
\begin{equation*}
LNM_n\forall\vdash\varphi\qquad\text{iff}\qquad NM_n\models\varphi
\end{equation*}
\end{theorem}
\begin{proof}
The soundness follows from the chain-completeness for axiomatic extensions of MTL$\forall$ (see \cite{mtl}).

For the completeness, note that each LNM$_n$-chain has at most $n$ elements (this follows from the axiomatization of LNM$_n$ and \Cref{teo:fixpointeq}). Moreover, it easy to see that every LNM$_n$-chain embeds into NM$_n$ preserving all $\inf$ and $\sup$.
To conclude, from chain completeness theorems and the previous results we have that if $LNM_n\forall\nvdash\varphi$, then $NM_n\not\models\varphi$.
\end{proof}
For the case of the infinite NM-chains, we need some other machinery. 

We now introduce a translation $^*$ between first-order formulas, and we will show that, given a Gödel-chain $\mathcal{A}$ and a formula $\varphi$,  $\mathcal{A}\models\varphi$ if and only if $\mathcal{A}_\text{NM}\models\varphi^*$. This fact will be fundamental to show some undecidability results, for some of the infinite NM-chains discussed in this paper. For one of them ($NM{'}^-_\infty$), however, the decidability remains an open problem.
\begin{definition}\label{def:trad}	
Let $\varphi$ be a formula. We define $\varphi^*$, inductively, as follows:
\begin{itemize}
\item If $\varphi$ is atomic, then $\varphi^*\df\varphi^2$.
\item If $\varphi\df\bot$, then $\varphi^*\df\bot$.
\item If $\varphi\df\psi\land\chi$, then $\varphi^*\df\psi^*\land\chi^*$.
\item If $\varphi\df\psi\&\chi$, then $\varphi^*\df\psi^*\&\chi^*$.
\item If $\varphi\df\psi\rightarrow\chi$, then $\varphi^*\df(\psi^*\rightarrow\chi^*)^2$.
\item If $\varphi\df(\forall x)\chi$, then $\varphi^*\df((\forall x)\chi^*)^2$.
\end{itemize}
\end{definition}	
\begin{lemma}\label{lem:m+}
Let $\varphi, \mathcal{A}, \mathbf{M}=\lag M, \lag m_c\rog_{c\in\mathbf{C}},\lag r_P\rog_{P\in\mathbf{P}}\rog$ be a formula, an NM-chain (call $\mathcal{A}^+$ the set of its positive elements) and a safe $\mathcal{A}$-model. Construct an $\mathcal{A}$-model $\mathbf{M}^+=\lag M, \lag m_c\rog_{c\in\mathbf{C}},\lag r'_P\rog_{P\in\mathbf{P}}\rog$ such that, for every evaluation $v$ and atomic formula $\psi$
\begin{align*}
\Vert\psi\Vert^\mathcal{A}_{\mathbf{M}^+,v}=\begin{cases}
\Vert\psi\Vert^\mathcal{A}_{\mathbf{M},v}&\text{if }\Vert\psi\Vert^\mathcal{A}_{\mathbf{M},v}\in\mathcal{A}^+\\
0&\text{otherwise}.
\end{cases}
\end{align*}
Then $\Vert\varphi^*\Vert^\mathcal{A}_{\mathbf{M},v}=\Vert\varphi^*\Vert^\mathcal{A}_{\mathbf{M}^+,v}$, for every $v$.
\end{lemma}
\begin{proof}
By structural induction over $\varphi$: if $\varphi\df\bot$ the claim is immediate. If $\varphi$ is atomic, then $\varphi^*\df\varphi^2$ and the claim easily follows from the definition of $\mathbf{M}^+$.

If $\varphi\df\psi\circ\chi$, with $\circ\in\{\land, \&, \rightarrow\}$, then the claim follows from the induction hypothesis over $\psi$ and $\chi$.

Finally, if $\varphi\df(\forall x)\chi$, then by the induction hypothesis $\Vert\chi^*\Vert^\mathcal{A}_{\mathbf{M},w}=\Vert\chi^*\Vert^\mathcal{A}_{\mathbf{M}^+,w}$, for every $w\equiv_x v$ and hence $\Vert\varphi^*\Vert^\mathcal{A}_{\mathbf{M},v}=\Vert\varphi^*\Vert^\mathcal{A}_{\mathbf{M}^+,v}$.
\end{proof}
\begin{theorem}\label{teo:mfix}
Let $\varphi$ be a formula and $\mathcal{A}$ be an NM-chain.
\begin{enumerate}
\item\label{m+eq} $\mathcal{A}\models\varphi^*$ iff $\Vert\varphi^*\Vert^\mathcal{A}_{\mathbf{M}^+,v}$, for every safe $\mathcal{A}$-model $\mathbf{M}$ and evaluation $v$.
\item Let $\mathcal{B}$ be a \emph{complete} NM-chain without negation fixpoint: call $\mathcal{B}^f$ its version with negation fixpoint $f$. It holds that
\begin{equation*}
\mathcal{B}\models\varphi^*\quad\text{iff}\quad \mathcal{B}^f\models\varphi^*.
\end{equation*}
\end{enumerate}
\end{theorem}
\begin{proof}
\begin{enumerate}
\item Immediate from \Cref{lem:m+}.
\item Due to \ref{m+eq} it is enough to check that $\Vert\psi\Vert^{\mathcal{B}^f}_{\mathbf{M}^+,v}\neq f$, for every formula $\psi$ and every $\mathcal{A}$-model $\mathbf{M}$ and evaluation $v$. This can be done by induction over $\psi$.
\begin{itemize}
\item If $\psi$ is atomic or $\bot$ the claim is immediate.
\item If $\psi\df\theta\circ\chi$, with $\circ\in\{\land, \&, \rightarrow\}$, then the claim follows from the induction hypothesis over $\theta$ and $\chi$.
\item Finally, if $\psi\df(\forall x)\chi$, then by the induction hypothesis $\Vert\chi\Vert^{\mathcal{B}^f}_{\mathbf{M},w}\neq f$, for every $w\equiv_x v$: if $\Vert\chi\Vert^{\mathcal{B}^f}_{\mathbf{M},w} < f$, for some $w\equiv_x v$, then $\Vert(\forall x)\chi\Vert^{\mathcal{B}^f}_{\mathbf{M},v}< f$.

    Suppose that $\Vert\chi\Vert^{\mathcal{B}^f}_{\mathbf{M},w}> f$, for every $w\equiv_x v$: moreover, by contradiction, assume that $\Vert(\forall x)\chi\Vert^{\mathcal{B}^f}_{\mathbf{M},v}= \inf_{w\equiv_x v}\{\Vert\chi\Vert^{\mathcal{B}^f}_{\mathbf{M},w}\}= f$. This means that the set of positive elements of $\mathcal{B}$ does not have infimum, in contrast with the hypothesis that $\mathcal{B}$ is complete.
\end{itemize}
\end{enumerate}
\end{proof}
\begin{theorem}\label{teo:tradgnm}
Let $\varphi$ be a formula, and $\mathcal{A}$ be a Gödel chain. Consider a safe $\mathcal{A}$-model $\mathbf{M}=\lag M, \lag m_c\rog_{c\in\mathbf{C}}, \lag r_P\rog_{P\in\mathbf{P}}\rog$: construct an $\mathcal{A}_\text{NM}$-model $\mathbf{M}'=\lag M, \lag m_c\rog_{c\in\mathbf{C}}, \lag r'_P\rog_{P\in\mathbf{P}}\rog$ such that, for every evaluation $v$ and \emph{atomic} formula $\psi$
\begin{equation*}
\Vert\psi\Vert_{\mathbf{M},v}^\mathcal{A}=\Vert\psi\Vert_{\mathbf{M}',v}^{\mathcal{A}_\text{NM}}.
\end{equation*}
Then for every evaluation $v$ we have
\begin{equation*}
\Vert\varphi\Vert_{\mathbf{M},v}^\mathcal{A}=\Vert\varphi^*\Vert_{\mathbf{M}',v}^{\mathcal{A}_\text{NM}}.
\end{equation*}
\end{theorem}
\begin{proof}
By structural induction over $\varphi$.
\begin{itemize}
\item If $\varphi$ is $\bot$ or atomic, then the claim is immediate.
\item $\varphi\df\psi\circ\chi$, with $\circ\in\{\land,\&\}$ and the claim holds for $\psi$ and $\chi$. It follows that $\Vert\theta\Vert_{\mathbf{M},v}^\mathcal{A}=\Vert\theta^*\Vert_{\mathbf{M}',v}^{\mathcal{A}_\text{NM}}$, for every $v$ and with $\theta\in\{\psi,\chi\}$: noting that these values are $0$ or idempotent elements the claim follows.
\item $\varphi\df\psi\rightarrow\chi$ and the claim holds for $\psi$ and $\chi$: it follows that $\Vert\theta\Vert_{\mathbf{M},v}^\mathcal{A}=\Vert\theta^*\Vert_{\mathbf{M}',v}^{\mathcal{A}_\text{NM}}$, for every $v$ and with $\theta\in\{\psi,\chi\}$. As previously noted, these values are idempotent elements or $0$. Since $\varphi^*\df(\psi^*\rightarrow\chi^*)^2$, an easy check shows that $\Vert\varphi\Vert_{\mathbf{M},v}^\mathcal{A}=\Vert\varphi^*\Vert_{\mathbf{M}',v}^{\mathcal{A}_\text{NM}}$, for every $v$.
\item $\varphi\df(\forall x)\psi$ and the claim holds for $\psi$. We have that $\Vert\psi\Vert_{\mathbf{M},w}^\mathcal{A}=\Vert\psi^*\Vert_{\mathbf{M}',w}^{\mathcal{A}_\text{NM}}$, for every $w$: if there is $w\equiv_x v$ such that $\Vert\psi\Vert_{\mathbf{M},w}^\mathcal{A}=0$, then the claim is immediate.

    Suppose that $\Vert\psi\Vert_{\mathbf{M},w}^\mathcal{A}>0$, for every $w\equiv_x v$.
    \subitem If $\Vert(\forall x)\psi\Vert_{\mathbf{M},v}^\mathcal{A}>0$, then $\Vert(\forall x)\psi\Vert_{\mathbf{M},v}^\mathcal{A}=\Vert(\forall x)\psi^*\Vert_{\mathbf{M}',v}^{\mathcal{A}_\text{NM}}=\Vert((\forall x)\psi^*)^2\Vert_{\mathbf{M}',v}^{\mathcal{A}_\text{NM}}=\Vert\varphi^*\Vert_{\mathbf{M}',v}^{\mathcal{A}_\text{NM}}$.
    \subitem If $\Vert(\forall x)\psi\Vert_{\mathbf{M},v}^\mathcal{A}=0$, then $\Vert(\forall x)\psi^*\Vert_{\mathbf{M}',v}^{\mathcal{A}_\text{NM}}=f$ and $\Vert((\forall x)\psi^*)^2\Vert_{\mathbf{M}',v}^{\mathcal{A}_\text{NM}}=\Vert\varphi^*\Vert_{\mathbf{M}',v}^{\mathcal{A}_\text{NM}}=0$.
\end{itemize}
\end{proof}
\begin{corollary}\label{cor:gnm}
Let $\varphi$ be a formula, $\mathcal{A}$ be a Gödel chain. We have that
\begin{equation*}
\mathcal{A}\models\varphi\quad\text{iff}\quad\mathcal{A}_\text{NM}\models\varphi^*.
\end{equation*}
\end{corollary}
\begin{proof}
An easy consequence of \Cref{teo:mfix,teo:tradgnm}.
\end{proof}
Recall that a subset of $[0,1]$ is complete if and only if it is compact with respect to the order topology (see for example \cite{ss}). Now, in \cite{fgod} it is showed that
\begin{theorem}[\cite{fgod}]\label{teo:godund}
Let $\mathcal{A}$ be a countable topologically closed subalgebra of $[0,1]_\text{G}$ (i.e. a countable complete subalgebra of $[0,1]_\text{G}$). Then $TAUT_\mathcal{A}\forall$ is not recursively enumerable.
\end{theorem}
In our case, we have
\begin{theorem}\label{teo:nmund}
Let $\mathcal{A}$ be a countable topologically closed subalgebra of $[0,1]_\text{NM}$ (i.e. a countable complete subalgebra of $[0,1]_\text{NM}$). Then $TAUT_\mathcal{A}\forall$ is not recursively enumerable.
\end{theorem}
\begin{proof}
Let $\mathcal{A}$ be a countable complete subalgebra of $[0,1]_\text{NM}$. 

If $\mathcal{A}$ has negation fixpoint then, due to the observations of \Cref{rem:nmcomp}, we can easily find a countable complete Gödel chain $\mathcal{B}$ such that $\mathcal{B}_\text{NM}\simeq \mathcal{A}$. From \Cref{teo:tradgnm} we have that $\varphi\in TAUT_\mathcal{B}\forall$ if and only if $\varphi^*\in TAUT_\mathcal{A}\forall$: since $TAUT_\mathcal{B}\forall$ is not recursively enumerable (\Cref{teo:godund}), then the same holds for $TAUT_\mathcal{A}\forall$.

If $\mathcal{A}$ does not have negation fixpoint, from \Cref{teo:mfix} we have that $\varphi^*\in TAUT_\mathcal{A}\forall$ if and only if $\varphi^*\in TAUT_{\mathcal{A}^f}\forall$, for every $\varphi$. Applying the argument of the previous case to $\mathcal{A}^f$, we have the theorem.
\end{proof}
\begin{corollary}
For $\mathcal{A}\in\{NM_\infty, NM^-_\infty, NM'_\infty\}$, $TAUT_\mathcal{A}\forall$ is not recursively enumerable.
\end{corollary}
\begin{problem}\label{prob:2}
Which is the arithmetical complexity of $TAUT_{NM{'}^-_\infty}\forall$ ? Is it recursively axiomatizable ?
\end{problem}
\subsubsection{Monadic fragments}\label{subsec:mon}
In this section, we analyze the the (un)decidability status of the validity problem for the monadic fragments associated to the complete subalgebras of $[0,1]_\text{NM}$, as well as the four infinite NM-chains hitherto discussed. Recall that in monadic first-order logic the language contains unary predicates and no functions or constant symbols.

Let $\mathcal{A}$ be a subalgebra of $[0,1]_\text{NM}$ (or of $[0,1]_\text{G}$): with $monTAUT_\mathcal{A}\forall$ we indicate the monadic tautologies associated to $\mathcal{A}$.

In \cite[theorem 1]{mgod} it is showed that the monadic fragment of finite Gödel chains is decidable, but as noted in the subsequent remark, the proof applies to the monadic fragments of arbitrary finite-valued logics. As a consequence we have
\begin{theorem}
Let $\mathcal{A}$ be a finite NM-chain: we have that $monTAUT_\mathcal{A}\forall$ is decidable.
\end{theorem}
However, for the infinite case the situation is more difficult; indeed
\begin{theorem}[{\cite{mgod}}]\label{teo:mongodund}
Let $\mathcal{A}$ be an infinite complete subalgebra of $[0,1]_\text{G}$: with the possible exception of $\mathcal{A}=G_\uparrow$, $monTAUT_\mathcal{A}\forall$ is undecidable.
\end{theorem}
Moving to the NM case, we obtain
\begin{theorem}
Let $\mathcal{A}$ be an infinite complete subalgebra of $[0,1]_\text{NM}$: with the possible exception of $\mathcal{A}\in\{NM_\infty, NM^-_\infty\}$, $monTAUT_\mathcal{A}\forall$ is undecidable.
\end{theorem}
\begin{proof}
Let $\mathcal{A}$ be a complete subalgebra of $[0,1]_\text{NM}$, with $\mathcal{A}\notin\{NM_\infty, NM^-_\infty\}$.

If $\mathcal{A}$ has negation fixpoint then, due to the observations of \Cref{rem:nmcomp}, we can easily find a complete Gödel chain $\mathcal{B}$, subalgebra of $[0,1]_\text{G}$, such that $\mathcal{B}_\text{NM}\simeq \mathcal{A}$: since $\mathcal{A}\neq NM_\infty$ then $\mathcal{B}\neq G_\uparrow$. From \Cref{teo:tradgnm} we have that $\varphi\in TAUT_\mathcal{B}\forall$ if and only if $\varphi^*\in TAUT_\mathcal{A}\forall$: since $TAUT_\mathcal{B}\forall$ is undecidable (\Cref{teo:mongodund}), then the same holds for $TAUT_\mathcal{A}\forall$.

If $\mathcal{A}$ does not have negation fixpoint, from \Cref{teo:mfix} we have that $\varphi^*\in TAUT_\mathcal{A}\forall$ if and only if $\varphi^*\in TAUT_{\mathcal{A}^f}\forall$, for every $\varphi$. Applying the argument of the previous case to $\mathcal{A}^f$, we have the theorem.
\end{proof}
\begin{corollary}
$monTAUT_{NM'_\infty}\forall$ is undecidable.
\end{corollary}
\begin{problem}\label{prob:3}
For $\mathcal{A}\in\{NM_\infty, NM^-_\infty, {NM'}^-_\infty\}$, is $monTAUT_\mathcal{A}\forall$ decidable ?
\end{problem}
\section{Open Problems}
Inspired by the work done in \cite{fgod,mgod} for (first-order) Gödel logics, in this paper we investigated the first-order tautologies associated with particular NM-chains: moreover, we have showed some decidability and undecidability results, for the full logics and the monadic case. 

Many questions are still open. The main one is the search for a full classification, in analogy with the one done for Gödel logics in \cite{fgod,mgod,pre}, of the (existence of) first-order logics associated to the various subalgebras of $[0,1]_\text{NM}$: for the subalgebras whose set of first-order tautologies is not recursively axiomatizable, instead, it could be studied its arithmetical complexity (for Gödel logics this has been done in \cite{fgod,mgod,pre,hajgod,hajgod1,hajgod2}).  Another theme that has not been analysed here concerns the (first-order) satisfiability problem about the subalgebras of $[0,1]_\text{NM}$ (for Gödel logics this has been done in \cite{bcp}).

We now discuss two more technical (and specific) problems, previously introduced in this paper.

\Cref{prob:2} is particularly interesting: if $TAUT_{NM{'}^-_\infty}\forall$ will result recursively axiomatizable, then the next step will be the search for a first-order logic complete with respect to  $NM{'}^-_\infty$. This logic could be a relevant infinite-valued logic, because $NM{'}^-_\infty$ satisfies the quantifiers shifting rules and hence we could work with formulas in prenex normal form.

Finally, consider \Cref{prob:3}: for $NM{'}^-_\infty$, this is a particular case of \cref{prob:2}. In the case in which $\mathcal{A}\in\{NM_\infty, NM^-_\infty\}$, instead, the solution is strictly connected with the analogous problem for Gödel logic with the chain $G_\uparrow$.

\paragraph{Acknowledgements}
The author would like to thank to Professor Stefano Aguzzoli for the suggestions, in particular concerning the translation from $\varphi$ to $\varphi^*$ introduced in \Cref{def:trad}.
\bibliography{NMfirst-order}

\newcommand{\etalchar}[1]{$^{#1}$}
\providecommand{\bysame}{\leavevmode\hbox to3em{\hrulefill}\thinspace}
\providecommand{\MR}{\relax\ifhmode\unskip\space\fi MR }
\providecommand{\MRhref}[2]{%
  \href{http://www.ams.org/mathscinet-getitem?mr=#1}{#2}
}
\providecommand{\href}[2]{#2}
\begin{thebibliography}{EZLM09}

\bibitem[ABM07]{abm}
S.~Aguzzoli, M.~Busaniche, and V.~Marra, \emph{{Spectral Duality for Finitely
  Generated Nilpotent Minimum Algebras, with Applications}}, J. Log. Comput.
  \textbf{17} (2007), no.~4, 749--765,
  \href{http://dx.doi.org/10.1093/logcom/exm021}{doi:10.1093/logcom/exm021}.

\bibitem[ABM09]{tbl}
S.~Aguzzoli, M.~Bianchi, and V.~Marra, \emph{{A temporal semantics for Basic
  Logic}}, Studia Logica \textbf{92} (2009), no.~2, 147--162,
  \href{http://dx.doi.org/10.1007/s11225-009-9192-3}{doi:10.1007/s11225-009-9192-3}.

\bibitem[AG10]{ag}
S.~Aguzzoli and B.~Gerla, \emph{{Probability Measures in the Logic of Nilpotent
  Minimum}}, Studia Logica \textbf{94} (2010), 151--176,
  \href{http://dx.doi.org/10.1007/s11225-010-9228-8}{doi:10.1007/s11225-010-9228-8}.

\bibitem[AGM05]{agm}
S.~Aguzzoli, B.~Gerla, and C.~Manara, \emph{{Poset Representation for Gödel and
  Nilpotent Minimum Logics}}, {Symbolic and Quantitative Approaches to
  Reasoning with Uncertainty} (L.~Godo, ed.), Lecture Notes in Computer
  Science, vol. 3571, Springer Berlin / Heidelberg, 2005,
  \href{http://dx.doi.org/10.1007/10.1007/11518655\_56}{doi:10.1007/10.1007/11518655\_56},
  pp.~469--469.

\bibitem[AGM08]{tmpg}
S.~Aguzzoli, B.~Gerla, and V.~Marra, \emph{Embedding {G\"{o}}del propositional
  logic into {P}rior's tense logic}, Proceedings of IPMU'08 (Torremolinos
  (M\'{a}laga)) (L.~Magdalena, M.~Ojeda-Aciego, and J.L. Verdegay, eds.), June
  2008,
  \url{http://www.gimac.uma.es/ipmu08/proceedings/papers/132-AguzzoliEtAl.pdf},
  pp.~992--999.

\bibitem[BC10]{bc}
M.~Busaniche and R.~Cignoli, \emph{{Constructive Logic with Strong Negation as
  a Substructural Logic}}, J. Log. Comput. \textbf{20} (2010), no.~4, 761--793,
  \href{http://dx.doi.org/10.1093/logcom/exn081}{doi:10.1093/logcom/exn081}.

\bibitem[BCF07]{mgod}
M.~Baaz, A.~Ciabattoni, and C.~G. Fermüller, \emph{{Monadic Fragments of Gödel
  Logics: Decidability and Undecidability Results}}, Logic for Programming,
  Artificial Intelligence, and Reasoning - 14th International Conference, LPAR
  2007, Yerevan, Armenia, October 15-19, 2007. Proceedings (N.~Dershowitz and
  A.~Voronkov, eds.), Lecture Notes in Computer Science, vol. 4790/2007,
  Springer Berlin / Heidelberg, 2007,
  \href{http://dx.doi.org/10.1007/978-3-540-75560-9}{doi:10.1007/978-3-540-75560-9},
  pp.~77--91.

\bibitem[BCP09]{bcp}
M.~Baaz, A.~Ciabattoni, and N.~Preining, \emph{{SAT in Monadic Gödel Logics: A
  Borderline between Decidability and Undecidability}}, {Logic, Language,
  Information and Computation} (H.~Ono, M.~Kanazawa, and R.~de~Queiroz, eds.),
  {Lecture Notes in Computer Science}, vol. 5514, Springer Berlin / Heidelberg,
  2009,
  \href{http://dx.doi.org/10.1007/978-3-642-02261-6\_10}{doi:10.1007/978-3-642-02261-6\_10},
  pp.~113--123.

\bibitem[BEG99]{beg}
D.~Boixader, F.~Esteva, and L.~Godo, \emph{On the continuity of t-norms on
  bounded chains}, {Proceedings of the 8th IFSA World Congress IFSA'99}
  (Taipei, Taiwan), August 1999, pp.~476--479.

\bibitem[BLZ96]{tgod}
M.~Baaz, A.~Leitsch, and R.~Zach, \emph{{Incompleteness of a first-order Gödel
  logic and some temporal logics of programs}}, Computer Science Logic - 9th
  International Workshop, CSL '95 Annual Conference of the EACSL Paderborn,
  Germany, September 22-29, 1995 Selected Papers (H.~K. Büning, ed.), Lecture
  Notes in Computer Science, vol. 1092/1996, Springer Berlin / Heidelberg,
  1996,
  \href{http://dx.doi.org/10.1007/3-540-61377-3\_28}{doi:10.1109/10.1007/3-540-61377-3\_28},
  pp.~1--15.

\bibitem[BP89]{bp}
W.~Blok and D.~Pigozzi, \emph{Algebraizable logics}, vol.~77, Memoirs of The
  American Mathematical Society, no. 396, American Mathematical Society, 1989,
  ISBN:0-8218-2459-7 - Available on
  \url{http://orion.math.iastate.edu/dpigozzi/}.

\bibitem[BPZ07]{fgod}
M.~Baaz, N.~Preining, and R.~Zach, \emph{{First-order Gödel logics}}, Ann.
  Pure. Appl. Logic \textbf{147} (2007), no.~1-2, 23--47,
  \href{http://dx.doi.org/10.1016/j.apal.2007.03.001}{doi:10.1016/j.apal.2007.03.001}.

\bibitem[Bus06]{bus}
M.~Busaniche, \emph{{Free nilpotent minimum algebras}}, Math. Log. Q.
  \textbf{52} (2006), no.~3, 219--236,
  \href{http://dx.doi.org/10.1002/malq.200510027}{doi:10.1002/malq.200510027}.

\bibitem[CH10]{pred}
P.~Cintula and P.~H\'{a}jek, \emph{Triangular norm predicate fuzzy logics},
  Fuzzy Sets Syst. \textbf{161} (2010), no.~3, 311--346,
  \href{http://dx.doi.org/10.1016/j.fss.2009.09.006}{doi:10.1016/j.fss.2009.09.006}.

\bibitem[CT06]{ct}
R.~Cignoli and P.~Torrens, \emph{{Free Algebras in Varieties of Glivenko
  MTL-algebras Satisfying the Equation $2(x^2) = (2x)^2$}}, Studia Logica
  \textbf{83} (2006), no.~1-3, 157--181,
  \href{http://dx.doi.org/10.1007/s11225-006-8302-8}{doi:10.1007/s11225-006-8302-8}.

\bibitem[DM71]{dm}
J.~M. Dunn and R.~K. Meyer, \emph{{Algebraic Completeness Results for Dummett's
  LC and Its Extensions}}, Math. Log. Q. \textbf{17} (1971), no.~1, 225--230,
  \href{http://dx.doi.org/10.1002/malq.19710170126}{doi:10.1002/malq.19710170126}.

\bibitem[Dum59]{dum}
M.~Dummett, \emph{A propositional calculus with denumerable matrix}, J. Symb.
  Log. \textbf{24} (1959), no.~2, 97--106,
  \url{http://www.jstor.org/stable/2964753}.

\bibitem[EG01]{mtl}
F.~Esteva and L.~Godo, \emph{{Monoidal t-norm based logic: Towards a logic for
  left-continuous t-norms}}, Fuzzy Sets Syst. \textbf{124} (2001), no.~3,
  271--288,
  \href{http://dx.doi.org/10.1016/S0165-0114(01)00098-7}{doi:10.1016/S0165-0114(01)00098-7}.

\bibitem[EGN06]{egnfirst}
F.~Esteva, L.~Godo, and C.~Noguera, \emph{{On rational weak nilpotent minimum
  logics}}, J. Mult.-Valued Logic Soft Comput. \textbf{12} (2006), 9--32.

\bibitem[EGN09]{egn2}
\bysame, \emph{{First-order t-norm based fuzzy logics with truth-constants:
  Distinguished semantics and completeness properties}}, Ann. Pure Appl. Log.
  \textbf{161} (2009), no.~2, 185--202,
  \href{http://dx.doi.org/10.1016/j.apal.2009.05.014}{doi:10.1016/j.apal.2009.05.014}.

\bibitem[EGN10]{egn1}
\bysame, \emph{{Expanding the propositional logic of a t-norm with
  truth-constants: completeness results for rational semantics}}, Soft Comput.
  \textbf{14} (2010), 273--284,
  \href{http://dx.doi.org/10.1007/s00500-009-0402-8}{doi:10.1007/s00500-009-0402-8}.

\bibitem[EZLM09]{nmcomp}
M.~El-Zekey, W.~Lotfallah, and N.~Morsi, \emph{Computational complexities of
  axiomatic extensions of monoidal t-norm based logic}, Soft Comput.
  \textbf{13} (2009), 1089--1097,
  \href{http://dx.doi.org/10.1007/s00500-008-0382-0}{doi:10.1007/s00500-008-0382-0}.

\bibitem[FGN10]{egn}
E.~Francesc, L.~Godo, and C.~Noguera, \emph{{On expansions of WNM t-norm based
  logics with truth-constants}}, Fuzzy Sets Syst. \textbf{161} (2010), no.~3,
  347--368,
  \href{http://dx.doi.org/10.1016/j.fss.2009.09.002}{doi:10.1016/j.fss.2009.09.002}.

\bibitem[FJG{\etalchar{+}}01]{colgod}
S.~Feferman, J.~W.~Dawson Jr., W.~Goldfarb, C.~Parsons, and W.~Sieg (eds.),
  \emph{Kurt gödel collected works}, paperback ed., vol. {1 Publications:
  1929-1936}, {Oxford University Press}, 2001,
  \href{http://www.oup.com/uk/catalogue/?ci=9780195147209}{ISBN:9780195147209}.

\bibitem[Fod95]{fod}
J.~Fodor, \emph{{Nilpotent minimum and related connectives for fuzzy logic}},
  {Fuzzy Systems, 1995. International Joint Conference of the Fourth IEEE
  International Conference on Fuzzy Systems and The Second International Fuzzy
  Engineering Symposium., Proceedings of 1995 IEEE International Conference
  on}, {IEEE}, {1995},
  \href{http://dx.doi.org/10.1109/FUZZY.1995.409964}{doi:10.1109/FUZZY.1995.409964},
  pp.~{2077--2082}.

\bibitem[Göd32]{godint}
K.~Gödel, \emph{{Zum intuitionistischen Aussagenkalkul}}, Anzeiger Akademie der
  Wissenschaften Wien \textbf{69} (1932), 65--66, Reprinted in \cite{colgod}.

\bibitem[Gis03]{gis}
J.~Gispert, \emph{Axiomatic extensions of the nilpotent minimum logic}, Rep.
  Math. Logic \textbf{37} (2003), 113--123,
  \url{http://www.iphils.uj.edu.pl/rml/rml-37/7-gispert.pdf}.

\bibitem[Háj02a]{haj}
P.~Hájek, \emph{Metamathematics of fuzzy logic}, paperback ed., Trends in
  Logic, vol.~4, Kluwer Academic Publishers, 2002,
  \href{http://www.springer.com/philosophy/logic/book/978-1-4020-0370-7}{ISBN:9781402003707}.

\bibitem[Háj02b]{pmtl}
\bysame, \emph{{Observations on the monoidal t-norm logic}}, Fuzzy Sets Syst.
  \textbf{132} (2002), no.~1, 107--112,
  \href{http://dx.doi.org/10.1016/S0165-0114(02)00057-X}{doi:10.1016/S0165-0114(02)00057-X}.

\bibitem[Háj05]{hajgod2}
\bysame, \emph{{A non-arithmetical Gödel logic}}, Log. J. IGPL \textbf{13}
  (2005), no.~4, 435--441,
  \href{http://dx.doi.org/10.1093/jigpal/jzi033}{doi:10.1093/jigpal/jzi033}.

\bibitem[Háj10a]{hajgod}
\bysame, \emph{{Arithmetical complexity of fuzzy predicate logics - A survey
  II}}, Ann. Pure Appl. Log. \textbf{161} (2010), no.~2, 212--219,
  \href{http://dx.doi.org/10.1016/j.apal.2009.05.015}{doi:10.1016/j.apal.2009.05.015}.

\bibitem[Háj10b]{hajgod1}
\bysame, \emph{{On witnessed models in fuzzy logic III - witnessed Gödel
  logics}}, Math. Log. Q. \textbf{56} (2010), no.~2, 171--174,
  \href{http://dx.doi.org/10.1002/malq.200810047}{doi:10.1002/malq.200810047}.

\bibitem[Jen03]{jen}
S.~Jenei, \emph{{On the structure of rotation-invariant semigroups}}, Arch.
  Math. Log. \textbf{42} (2003), no.~5, 489--514,
  \href{http://dx.doi.org/10.1007/s00153-002-0165-8}{doi:10.1007/s00153-002-0165-8}.

\bibitem[KMP00]{kmp}
E.P. Klement, R.~Mesiar, and E.~Pap, \emph{Triangular norms}, hardcover ed.,
  Trends in Logic, vol.~8, Kluwer Academic Publishers, 2000,
  \href{http://www.springer.com/philosophy/logic/book/978-0-7923-6416-0}{ISBN:978-0-7923-6416-0}.

\bibitem[Nog06]{nog}
C.~Noguera, \emph{{Algebraic study of axiomatic extensions of triangular norm
  based fuzzy logics}}, Ph.D. thesis, IIIA-CSIC, 2006, Available on
  \url{http://www.carlesnoguera.cat/files/NogueraPhDThesis.pdf}.

\bibitem[Pre03]{pre}
N.~Preining, \emph{{Complete recursive axiomatizability of Gödel logics}},
  Ph.D. thesis, Vienna University of Technology, Austria, 2003, Available on
  \url{http://www.logic.at/staff/preining/pubs/phd.pdf}.

\bibitem[Pre10]{presurv}
\bysame, \emph{{Gödel Logics - A Survey}}, {Logic for Programming, Artificial
  Intelligence, and Reasoning} (C.~Fermüller and A.~Voronkov, eds.), Lecture
  Notes in Computer Science, vol. 6397, Springer Berlin / Heidelberg, 2010,
  \href{http://dx.doi.org/10.1007/978-3-642-16242-8\_4}{doi:10.1007/978-3-642-16242-8\_4},
  pp.~30--51.

\bibitem[SS95]{ss}
J.~A. Seebach and L.~A. Steen, \emph{Counterexamples in topology}, {reprint of
  1978} ed., Dover Publications, 1995, ISBN:048668735X.

\bibitem[SS05]{prob}
B.~Schweizer and A.~Sklar, \emph{{Probabilistic Metric Spaces}}, {updated
  reprint of 1983} ed., Dover Publications, 2005, ISBN:9780486445144.

\end{thebibliography}
\bibliographystyle{amsalpha}
\end{document}